\documentclass{article}

\usepackage{arxiv}

\usepackage[colorlinks,citecolor=blue,urlcolor=blue]{hyperref}
\usepackage[utf8]{inputenc} 
\usepackage[T1]{fontenc}    
\usepackage{booktabs}       
\usepackage{nicefrac}       
\usepackage{microtype}      
\usepackage{doi}

\usepackage[english]{babel}
\usepackage{amsmath}
\usepackage{amstext}
\usepackage{amsthm}
\usepackage{graphicx}%
\usepackage{amsfonts}%
\usepackage{amssymb}
\usepackage{dsfont}
\usepackage{rotating}
\usepackage{xy}
\usepackage{pdflscape}
\usepackage{url}
\usepackage[round,authoryear]{natbib}
\usepackage{mathrsfs}
\usepackage{setspace}
\usepackage{caption}
\usepackage{multirow}
\usepackage{enumitem}
\usepackage{xcolor}

\newcommand{\by}{\mathbf{y}}

\numberwithin{equation}{section}
\newtheorem{thm}{Theorem}[section]
\newtheorem{lemma}[thm]{Lemma}
\newtheorem{corollary}[thm]{Corollary}
\newtheorem{proposition}[thm]{Proposition}

\newtheorem{remark}[thm]{Remark}

\newcommand{\btheta}{\boldsymbol{\theta}}
\newcommand{\bphi}{\boldsymbol{\phi}}
\newcommand{\bEta}{\boldsymbol{\eta}}

\newcommand{\sumN}{\sum_{i=1}^n}
\newtheorem{condition}{Condition}

\newcommand{\LZadd}[1]{\textcolor{black}{#1}}

\newcommand{\BASadd}[1]{\textcolor{black}{#1}}

\newcommand{\given}{\,|\,}

\date{\today}	

\title{Asymptotic posterior normality of the generalized extreme value distribution}
\author{Likun Zhang \\
	Department of Statistics \\
	University of Missouri \\
	Columbia, MO, 65211 \\
	\And
	Benjamin A. Shaby \\
	Department of Statistics\\
	Colorado State University\\
	Fort Collins, CO, 80523
}

\renewcommand{\shorttitle}{Asymptotic posterior normality of GEV distribution}

\hypersetup{
pdftitle={Asymptotic posterior normality of GEV distribution},
pdfauthor={Likun Zhang, Benjamin A. Shaby},
pdfkeywords={non-regular parametric families; posterior consistency},
}

\begin{document}
\maketitle

\begin{abstract}
The univariate generalized extreme value (GEV) distribution is the most commonly used tool for analyzing the properties of rare events.  The ever greater utilization of Bayesian methods  for extreme value analysis warrants detailed theoretical investigation, which has thus far been underdeveloped. Even the most basic asymptotic results are difficult to obtain because the GEV fails to satisfy standard regularity conditions.  Here, we prove that the posterior distribution of the GEV parameter vector, given $n$ independent and identically distributed samples, converges in distribution to a trivariate normal distribution.  The proof necessitates analyzing integrals of the GEV likelihood function over the entire parameter space, which requires considerable care because the support of the GEV density depends on the parameters in complicated ways.
\end{abstract}

\keywords{non-regular parametric families; posterior consistency}

\section{Introduction}\label{sec:intro}
The family of generalized extreme value (GEV) distributions is a standard tool for studying the tail behavior of variables in fields ranging from finance to climate science, and Bayesian inference for the GEV family has become increasingly popular \citep[][e.g.]{coles1996bayesian, martins2000generalized,northrop2016posterior}.  Although the GEV was introduced almost a century ago \citep{fisher1928limiting}, classical asymptotic properties that hold for common parametric families such as the exponential families still remain to be confirmed under the GEV, which fails to meet standard regularity conditions. In this paper, we elucidate an important large-sample property of Bayesian inference based on independent GEV samples, and show that the posterior distribution, paired with a sufficiently regular prior, can be well-approximated by a normal distribution as the sample size $n$ increases.

Consider a sequence of independent samples $\by_n=(Y_1,\ldots,Y_n)$ from a generic parametric distribution $P_{\bphi}$, for $\bphi\in \Phi\subset \mathbb{R}^k$, with density $p(y\given \bphi)$ in regard to a dominating measure $\mathcal{P}$ on some measurable space $(\mathcal{X},\mathcal{A})$. Under common regularity conditions \citep[e.g.]{fisher1922mathematical,cramer1946mathematical,wald1949note}, the asymptotic normality of the maximum likelihood estimator (MLE) $\hat{\bphi}_n$, which maximizes the joint log-likelihood $L_n(\bphi)=\sumN \log p(Y_i\given \bphi)$ or solves its associated score equations $\nabla L_n(\bphi)=\boldsymbol{0}$, is well-established. Analogous results hold for Bayesian estimates. For parametric families that satisfy common regularity conditions, the information contained in the samples $Y_1,\ldots, Y_n$ overwhelms the prior density $\pi(\bphi)$ as $n\rightarrow\infty$, and the posterior density
\begin{equation}\label{eqn:posterior}
	\pi_n(\bphi)=\prod_{i=1}^n p(Y_i\given \bphi)\pi(\bphi)\Big/C_n,\quad C_n=\int_\Phi \prod_{i=1}^n p(Y_i\given \bphi)\pi(\bphi) d\bphi,
\end{equation}
will concentrate at $\hat{\bphi}_n$ in the form of a normal density \citep{bernstein1927theory,von1931wahrscheinlichkeit}.

However, as a result of the non-regularity of the GEV, frequentist asymptotic results have only recently been established, and Bayesian asymptotic behavior remains unexplored.  Treating $\boldsymbol{Y}_n$ as a sequence of maxima over finitely-sized blocks that will become GEV distributed as the block size goes to infinity,  \citet{dombry2015existence} proved the existence and strong consistency of the \textit{local} MLE that maximizes the GEV log-likelihood in a pre-specified closed neighborhood of the true parameter $\btheta_0$, for $\xi_0>-1$. \citet{bucher2018maximum} and \citet{dombry2019maximum} established the asymptotic normality of the local MLE when $\xi_0>-1/2$. Under a simpler setting where $\boldsymbol{Y}_n$ consists of independent samples from the GEV distribution, \citet{bucher2017maximum} proved a $O_p(1/\sqrt{n})$ rate of convergence of the local MLE to $\btheta_0$ and showed the central limit relations hold uniformly over the pre-specified compact set for $\xi_0>-1/2$.

However, interesting Bayesian quantities often require integrals over the entire parameter space. Therefore, theoretical analysis requires moving away from the vicinity of $\btheta_0$ and examining the likelihood function globally. \citet{zhang2020uniqueness} explored the global properties of $\prod_{i=1}^n p(Y_i\given \btheta)$ and showed that the local MLE defined in \citet{bucher2017maximum} actually gives the global maximum for the likelihood function when $n$ is sufficiently large. They showed that, informally, the likelihood function becomes highly peaked around the local MLE, and that it decays quickly as $||\btheta-\hat{\btheta}_n||\rightarrow\infty$. We will use this result to approximate the normalizing constant $C_n$ defined in \eqref{eqn:posterior}. 

If $\by_n$ is sampled from a GEV distribution, the elucidation of the aforementioned asymptotic properties is not trivial. The GEV \LZadd{cumulative distribution function} can be parameterized by $\btheta=(\tau,\mu,\xi)$:
\begin{equation*}
	P_{\btheta}(y)=
	\begin{cases}
		\exp\left\{-\left[1+\xi\left(\frac{y-\mu}{\tau}\right)\right]^{-1/\xi}\right\},&\xi\neq 0,\\
		\exp\left\{-\exp\left[-\frac{y-\mu}{\tau}\right]\right\}, &\xi=0,
	\end{cases}
\end{equation*}
for $1+\xi(y-\mu)/\tau>0$ when $\xi\neq 0$, with scale parameter $\tau>0$, location parameter $\mu \in\mathbb{R}$, and shape parameter $\xi\in\mathbb{R}$ \citep{mises1936distribution,jenkinson1955frequency}. Since the support of the GEV distribution depends on $\btheta$, it is especially challenging not only to establish the existence and consistency of its MLE, but also to derive the asymptotic normality of the posterior. For the remainder of the paper, we use $\bphi$ and $\btheta$ to denote the parameter vectors of a generic family of distributions and the GEV family, respectively.

In our main result on posterior asymptotic normality, we assume a class of prior distributions that is general enough to accommodate most cases of interest.  This assumption is not critical, however, as other classes of priors would yield identical results, after slight modification of the proofs.  Practical recommendations for choices of prior distributions is a nuanced topic that is well beyond the scope of the current work.

\section{Main results}
Our main result is establishing the asymptotic normality of the posterior distribution, given i.i.d. samples from the GEV distribution and a sufficiently regular prior distribution. 

Many versions of  
sufficient conditions for asymptotic posterior normality have appeared in the literature, including \citet{lecam1958proprietes}, \citet{freedman1963asymptotic}, \citet{walker1969asymptotic}, \citet{bickel1969some}, \citet{chao1970asymptotic} and \citet{van2000asymptotic}. While these conditions have different restrictions on the sample space $\mathcal{X}$, the parameter space $\Phi$, and $p(y\given \bphi)$, they mostly require the support $\{y:p(y\given \bphi)>0\}$ to be independent of $\bphi$, as is the case for the regularity conditions imposed by \citet{cramer1946mathematical} and \citet{wald1949note} for MLE. The GEV distribution clearly violates this assumption.

Our results remain valid under any fairly regular prior distribution.  To make this concrete, we define one of many possible classes of priors that is compatible with the theory we present below.
\begin{condition}\label{cond_for_prior}
	Let $\pi(\btheta)$ be any continuous proper prior density function with support on $\Theta=\{\btheta=(\tau,\mu,\xi):\tau>0,\mu\in\mathbb{R}, \xi>-1/2\}$ or 
	an improper prior density function that satisfies $ \pi(\btheta)= g(\xi)/\tau$, where $g$ is bounded on any interval $[-1/2,c]$, $c>-1/2$, and is regularly varying at infinity with index $\alpha\in \mathbb{R}$; that is, $g(t\xi)/g(\xi)\rightarrow t^{\alpha}$ as $\xi\rightarrow\infty$.
\end{condition}
We are now ready to state the following main result:
\begin{thm}\label{thm:main_result_posterior_normality}
	Suppose $Y_1, Y_2,\ldots$ are independently sampled from a GEV distribution $P_{\btheta_0}$ with $\btheta_0\in \Theta=\{\btheta:\tau>0, \xi>-1/2\}$, and $\hat{\btheta}_n$ is the local MLE based on $Y_1, \ldots, Y_n$. Let $L''_n(\hat{\btheta}_n)$ be the Hessian of the log-likelihood $L_n$ at $\hat{\btheta}_n$ and $\Sigma_n(\hat{\btheta}_n)=\{-L''_n(\hat{\btheta}_n)\}^{-1}$. Paired with a prior $\pi(\btheta)$ that satisfies Condition \ref{cond_for_prior}, 
	the posterior distribution of $\Sigma_n^{-1/2}(\hat{\btheta}_n)(\btheta-\hat{\btheta}_n)$ converges in distribution to a trivariate standard normal distribution. That is,
	\begin{equation}\label{eqn:posterior_normal}
		\mathcal{P}\left(\left. \Sigma_n^{-1/2}(\hat{\btheta}_n)(\btheta-\hat{\btheta}_n)\leq \boldsymbol{a}\,\right\rvert \boldsymbol{Y}_n\right)\rightarrow \Psi(\boldsymbol{a}), \quad n\rightarrow \infty,
	\end{equation}
	$P_{\btheta_0}$-almost surely for all $\boldsymbol{a}\in\mathbb{R}^3$, where $\Psi(\cdot)$ is the distribution function of the trivariate standard normal distribution.
\end{thm}

Again, Condition \ref{cond_for_prior} is sufficient but not necessary for asymptotic posterior normality, and other large classes of priors will also be sufficient. We impose this condition on the priors because it simplifies the proof and it is fairly common for any model with location and scale parameters. An important special case is the form $\pi(\tau,\mu,\xi)\propto 1/\tau$, which is invariant under reparametrization. Under the GEV model, the corresponding maximal data information (MDI) \citep{zellner1971introduction}, and Jeffreys \citep{jeffreys1961theory} priors can both be factorized into $g(\xi)/\tau$ \citep{northrop2016posterior}, although $g(\xi)$ increases without limit at $-1/2$ and $\infty$ for the Jeffreys prior and consequently the Jeffreys prior is not permissible for any sample size $n$. Hence the tail heaviness of $g(\xi)$ needs to be controlled to ensure posterior propriety and asymptotic normality. In \citet{zhang2024reference}, the performances of the ruled-based reference priors \citep{berger2009formal}, MDI priors and the Beta priors were examined extensively, and recommendations were also provided on the choice of prior for the GEV model according to the use case and the tail-heaviness of the observations.

For the priors that satisfy Condition \ref{cond_for_prior}, it is possible to study the properties of $C_n$ defined in \eqref{eqn:posterior}, and then verify non-standard sets of sufficient conditions for asymptotic posterior normality which admit non-regular families of densities and cope with challenges posed by the support. Examples of such results include \citet{dawid1970limiting}, who restricts $\Phi\subseteq \mathbb{R}$; \citet{heyde1979asymptotic}, who focus on stochastic processes with $\Phi\subseteq \mathbb{R}$;  \citet{chen1985asymptotic}, who extends $\Phi$ to be in $\mathbb{R}^k$; and \citet{ghosal1995convergence}, whose posterior convergence is more general and includes non-normal limits. In this paper, we will work with the conditions proposed by \citet{chen1985asymptotic} due to their generality and simplicity.

In Section~\ref{sec:chen_pos}, we describe the conditions in \citet{chen1985asymptotic} and explain heuristically their respective roles in the deduction of asymptotic posterior normality, and compare them with other sets of conditions in the literature. In Section~\ref{sec:proof}, we utilize the asymptotic properties of the likelihood function from \cite{zhang2020uniqueness} to prove the first two conditions of \citet{chen1985asymptotic}. We then verify the third and last condition, and carefully prove that the contribution to the integral of $C_n$ outside any pre-specified compact neighborhood \BASadd{is negligible} asymptotically. In Section~\ref{sec:discussion}, we summarize the ancillary results that we derived in the process of checking \citet{chen1985asymptotic}'s conditions, and discuss how the asymptotic posterior normality might be used to derive ruled-based noninformative priors for the family of GEV distributions.

\section{Sufficient conditions for asymptotic posterior normality}\label{sec:chen_pos}
\citet{chen1985asymptotic} established conditions that lead to asymptotic posterior normality for any generic family of distributions parameterized by $\bphi\in\Phi\in\mathbb{R}^k$. Denote $B_r(\bphi)=\{\bphi'\in \Phi: ||\bphi'-\bphi||<r\}$ as a neighborhood of $\bphi$, and $S(y)=\{\bphi\in\Theta:p(y\given \bphi)>0\}$ as the set of plausible $\bphi$ values under which $y$ can be observed. Given a sequence of samples $\boldsymbol{Y}_n=(Y_1,\ldots,Y_n)$, define
\begin{equation*}
	\Omega_n=\bigcap_{i=1}^n S(Y_i).
\end{equation*}
Then the domain of the integral in \eqref{eqn:posterior} can be reduced to $\Omega_n$. Recall $L_n(\bphi)=\sumN \log p(Y_i\given \bphi)$, and we see $\Omega_n$ is also the domain for the log-likelihood.

Most sufficient conditions in the literature for asymptotic posterior normality either require independence of the support from the parameters or have technical formulations that are difficult to verify. \LZadd{For example, Theorem 10.1 in \citet{van2000asymptotic} asks for the existence of a sequence of uniformly consistent test functions, which is difficult to verify due to the non-regularity of the GEV}. The conditions in \citet{chen1985asymptotic} are appealing because there are only three basic conditions that provide ample operational flexibility.

\begin{lemma}[Modified from \citet{chen1985asymptotic}]\label{lem:Chen_conditions}
	Suppose with probability $1$ for each $n>N$ ($N>0$), there exists a strict local maximum point $\hat{\bphi}_n$ for $L_n(\bphi)$ such that the gradient $L'_n(\hat{\bphi}_n)$ is a zero vector and the Hessian matrix $L''_n(\hat{\bphi}_n)$ is negative definite. Also suppose the prior density $\pi(\bphi)$ is positive and continuous at $\bphi=\bphi_0$, and $\hat{\bphi}_n$ tends to $\bphi_0$ almost surely as $n\rightarrow\infty$. The three basic conditions presented below will ensure that the posterior converges in distribution to a trivariate normal distribution as $n$ grows to infinity:
	\begin{enumerate}[label=(C\arabic*)]
		\item \label{C1} The largest eigenvalue of $\Sigma_n(\hat{\bphi}_n)=\{-L''_n(\hat{\bphi}_n)\}^{-1}$
		converges almost surely to $0$ as $n\rightarrow \infty$.
		\item \label{C2} For any $\epsilon>0$, there almost surely exists $N_\epsilon>0$ and $r>0$ such that, for all $n>\max\{N,N_\epsilon\}$ and $\bphi\in B_r(\hat{\bphi}_n)$, $L''_n(\bphi)$ exists and satisfies 
		\begin{equation*}
			I-A(\epsilon)\leq L''_n(\bphi)\{L''_n(\hat{\bphi}_n)\}^{-1}\leq I+A(\epsilon),
		\end{equation*}
		in which $I$ is the $k\times k$ identity matrix, $A(\epsilon)$ is a $k\times k$ positive semi-definite matrix whose largest eigenvalue tends to $0$ as $\epsilon\rightarrow 0$ and $k$ is the dimension of the space $\Phi$. Also, the Loewner order, i.e., $A \leq B$ if and only if $B-A$ is positive semidefinite, is used here.
		\item \label{C3} The posterior distribution $\pi_n(\bphi)$ asymptotically concentrates around $\hat{\bphi}_n$. That is, for any $r>0$,
		\begin{equation*}
			\int_{\Omega_n\setminus B_r(\hat{\bphi}_n)}\pi_n(\bphi)d\bphi\rightarrow 0\quad \text{a.s.}, \quad n\rightarrow\infty.
		\end{equation*}
	\end{enumerate}
	More specifically, under \ref{C1} and \ref{C2}, \ref{C3} is the necessary and sufficient condition that \eqref{eqn:posterior_normal} holds. 
\end{lemma}

The original statement of \citet{chen1985asymptotic} calculates the gradient and Hessian of the posterior $\pi_n(\bphi)$ and approximates the posterior distribution by $N(\boldsymbol{m}_n, \Sigma_n(\boldsymbol{m}_n))$, where $\boldsymbol{m}_n$ is the posterior local mode that solves $L'_n(\bphi)+\{\log \pi(\bphi)\}'=0$. However, his proof can easily be modified to apply to the MLE $\hat{\bphi}_n$, given that $\pi(\bphi)$ is continuous at $\bphi_0$ and $\hat{\bphi}_n$ is strongly consistent; see Appendix \ref{appendix:revised_proof} for the modified proof.

We caution that the seemingly short list of conditions in Lemma \ref{lem:Chen_conditions} implicitly relies on several further assumptions, which are trivially satisfied for many continuous parametric families and appear explicitly in other works. For example, assumption (C1) in \citet{dawid1970limiting} requires $\Omega_n$ to contain $\bphi_0$ in its interior. \citet{chen1985asymptotic} essentially assumes this condition when he applies a Taylor expansion to approximate $L_n(\bphi)$ near the mode $\hat{\bphi}_n$. Also, most sets of conditions assume the identifiability of the parameter; that is, different values of $\bphi$ imply different distributions, which induces positive Kullback-Leibler divergence and the consistency of the MLE. \citet{chen1985asymptotic} circumvents this condition by assuming the existence of local MLE and the negative definiteness of the Hessian at the local MLE. Fortunately, the GEV family possesses both the required identifiability and a consistent local MLE.

Now we provide an overview of the role of each condition and explain heuristically how the asymptotic posterior normality is attained. Firstly, \ref{C1} implies that the posterior density $\pi_n(\bphi)$ is highly peaked around $\hat{\bphi}_n$. This condition is also required by (A3) in \citet{heyde1979asymptotic}. The term $-L''_n(\hat{\bphi}_n)$ is often referred to as the observed Fisher information matrix. If $-n^{-1}L''_n(\hat{\bphi}_n)$ tends to $I(\bphi_0)$ almost surely, which is the case for many parametric families, then \ref{C1} holds automatically.

Secondly, the main function of \ref{C2} is to make sure $L''_n(\bphi)$ behaves sufficiently smoothly for values of $\bphi$ near $\hat{\bphi}_n$. This condition can also be found in (B4) of \citet{walker1969asymptotic}, (C11) of \citet{dawid1970limiting} and (A5) of \citet{heyde1979asymptotic}, while (IH1) in \citet{ghosal1995convergence} is weaker and only assumes Lipschitz continuity of $L_n(\bphi)$ in a pre-specified compact set. Expanding a Taylor series only to the first order under \ref{C2}, we can use the Lagrange's form of the remainder to obtain $L_n(\bphi)=L_n(\hat{\bphi}_n)+(\bphi-\hat{\bphi}_n)^TL''_n(\bphi^+_n)(\bphi-\hat{\bphi}_n)/2$ for some $\bphi_n^+$ lying between $\bphi$ and $\hat{\bphi}_n$. When $\bphi$ is near $\hat{\bphi}_n$, $L''_n(\bphi^+_n)\approx L''_n(\hat{\bphi}_n)$, and thus the posterior can be approximated as
\begin{equation}\label{eqn:posterior_taylor}
	\begin{split}
		\pi_n(\bphi)&= \pi_n(\hat{\bphi}_n)\exp\{\log \pi(\bphi)-\log \pi(\hat{\bphi}_n)\}\exp\{L_n(\bphi)-L_n(\hat{\bphi}_n)\}\\
		&\approx c(\by_n)\exp\left\{-\frac{1}{2}(\bphi-\hat{\bphi}_n)^T\{-L''_n(\hat{\bphi}_n)\}(\bphi-\hat{\bphi}_n)\right\},
	\end{split}
\end{equation}
in which the constant $c(\by_n)$ becomes independent of the prior, as 
\ref{C1} ensures that the observed information  $-L''_n(\hat{\bphi}_n)$ dominates the prior information $-\pi''(\hat{\bphi}_n)$, i.e., $|\pi''(\hat{\bphi}_n)|/|L''_n(\hat{\bphi}_n)|\rightarrow 0$ when $n\rightarrow \infty$. The approximation in \eqref{eqn:posterior_taylor} already confirms that $\pi_n(\bphi)$ behaves like a Gaussian kernel inside a small neighborhood of $\hat{\bphi}_n$. 

Thirdly, \ref{C3} concerns the global properties of $L_n$, which are needed to validate each approximation leading to \eqref{eqn:posterior_taylor} when moving away from $\hat{\bphi}_n$. In fact, under \ref{C1} and \ref{C2}, the consistency of the posterior, i.e. \ref{C3}, is sufficient and necessary for that purpose; see Lemma \ref{lem:bounded_concentration} of Appendix \ref{appendix:revised_proof}. In essence, \ref{C3} is equivalent to the assertion that $\pi_n(\bphi)$ converges in distribution to a distribution degenerate at $\bphi_0$, given the strong consistency of the local MLE. It is worth noting that the existence of a sequence of uniformly consistent tests required by \citet{van2000asymptotic} is used in effect to ensure \ref{C3} of \citet{chen1985asymptotic}; see the proof of Theorem 10.1 therein. Moreover, the need for \ref{C3} to assure the consistency of the Bayes estimators has also been stressed by \citet{freedman1963asymptotic} and \citet{diaconis1986consistency}. This property holds in many regular parametric families under mild conditions, but it is again difficult to verify for the family of GEV distributions. In comparison, Equation (5) in \citet{walker1969asymptotic}, (C7) in \citet{dawid1970limiting}, (A4) in \citet{heyde1979asymptotic} and (IH2) in \citet{ghosal1995convergence} all require uniform-boundedness of the tail behaviors in similar ways, and are stricter 
than \ref{C3}. For example, Equation (5) in \citet{walker1969asymptotic} states 
\begin{equation}\label{eqn:C3_1}
	\lim_{n\rightarrow \infty}\mathcal{P}\left[\sup_{\bphi\in \Omega_n\setminus B_r(\hat{\bphi}_n)} n^{-1}\{L_n(\bphi)-L_n(\hat{\bphi}_n)\}<-k(r)\right]=1,
\end{equation}
where $k(r)$ is a positive constant depending on $r$. Equation \eqref{eqn:C3_1} is equivalent to (C3.1) in \citet{chen1985asymptotic}, which is sufficient for \ref{C3} under \ref{C1} and \ref{C2}. Furthermore, \eqref{eqn:C3_1} implies that there is no likelihood value of $\bphi$ not near $\bphi_0$ that is larger than that of the local MLE $\hat{\bphi}_n$. Meanwhile for the GEV likelihood, $\hat{\bphi}_n$ is almost surely the unique maximum point in $B_r(\hat{\bphi}_n)$ for all sufficiently large $n$ \citep[Proposition 2]{dombry2015existence}. Therefore, if \eqref{eqn:C3_1} holds for the GEV family, the local maximizer $\hat{\btheta}_n$ is also the unique global maximizer.

Although \citet{zhang2020uniqueness} established the global optimality of $\hat{\btheta}_n$ for the GEV, they did not verify \eqref{eqn:C3_1} because it essentially requires global H\"{o}lder continuity of $n^{-1}L_n(\btheta)$. They only showed the weaker property that $L_n(\btheta)<L_n(\hat{\btheta}_n)$ for $\btheta\neq \hat{\btheta}_n$. However, we will use results contained therein to validate \ref{C1} and \ref{C2}, and then study the integral in \ref{C3} directly.

\section{Proof of asymptotic posterior normality for the GEV distributions}\label{sec:proof}
\subsection{Steepness and smoothness}\label{sec:cond_1_2}
Since the score equations of the GEV log-likelihood only have roots if $\xi_0>-1$ and second-order asymptotic properties needed for \ref{C2} only hold if $\xi_0>-1/2$, we confine the parameter space to be $\Theta=\{\btheta:\tau>0,\xi>-1/2\}$. Following the notation in \citet{zhang2020uniqueness}, we define $\beta=\beta(\btheta)=\mu-\tau/\xi$ and re-parameterize the GEV log-likelihood using the one-to-one mapping from $\btheta = (\tau, \mu,\xi)$ to $\bEta=(\tau,\beta,\xi)$. Under the $\bEta$-parameterization, the log-likelihood function can be written as
\begin{equation}\label{Log-lik}
	L_n (\bEta)=
	-n\log\tau+(\xi+1)\sumN \log W_i^{-1/\xi}(\bEta)-\sumN W_i^{-1/\xi}(\bEta),\quad \xi\neq 0,
\end{equation}
where 
\begin{equation}\label{wi}
	W_i(\bEta)=1+\xi\left(\frac{Y_i-\mu}{\tau}\right)=\frac{\xi}{\tau}(Y_i-\beta).
\end{equation}
The case $\xi=0$ can also be unified in this formulation by continuity, as $W_i^{-1/\xi}(\bEta)\rightarrow\exp\{-(Y_i-\mu)/\tau\}$ as $\xi\rightarrow 0$. The advantage of re-parameterizing from $(\tau,\mu,\xi)$ to $(\tau,\beta,\xi)$ in this way is that it helps delineate the domain of the log-likelihood concisely:
\begin{equation*}
	\Omega_n=\{\bEta:\tau>0,\xi>-1/2\text{ and }\xi(Y_i-\beta)>0,\;i=1,\ldots,n\}.
\end{equation*}
Consequently, $\Omega_n$, which is also the domain of integral in \eqref{eqn:posterior}, becomes two rectangular boxes on each side of the $\tau$-$\beta$ plane at $\xi=0$. Since the determinant of Jacobian for the transformation from $\btheta$ to $\bEta$ is an identity matrix, proving the asymptotic posterior normality under $\btheta$-parameterization and $\bEta$-parameterization are equivalent. Additionally, applying the substitution $\beta=\mu-\tau/\xi$ simplifies the integral in \eqref{eqn:posterior}:
\begin{equation*}
	\int_{\Theta\cap\{\xi>0\}}\pi(\btheta)\exp\{L_n(\btheta)\}d\btheta = \int_0^\infty\int_{-\infty}^{Y_{(1)}}\int_0^\infty \pi(\bEta)\exp\{L_n(\bEta)\} d\tau d\beta d\xi,
\end{equation*}
and 
\begin{equation*}
	\int_{\Theta\cap\{-1/2<\xi<0\}}\pi(\btheta)\exp\{L_n(\btheta)\}d\btheta = \int_{-1/2}^{0}\int_{Y_{(n)}}^{\infty}\int_0^\infty \pi(\bEta)\exp\{L_n(\bEta)\} d\tau d\beta d\xi,
\end{equation*}
in which $Y_{(1)}$ and $Y_{(n)}$ denote the sample minimum and maximum, respectively. Note that the prior $\pi(\cdot)$ defined by Condition \ref{cond_for_prior} does not depend on $\mu$ or $\beta$.

Appendix G in \citet{zhang2020uniqueness} lists the exact form of the Hessian matrix of $L_n(\bEta)$, in which every element can be written as linear combinations of sums of the form
\begin{equation*}
	\sumN W_i^{-k-a/\xi}(\bEta)\log^b W_i(\bEta),
\end{equation*}
where $k,b=0,1,2$, $a=0,1$. For \ref{C1}, we need to evaluate the Hessian $L''_n$ at the local MLE, whose limiting form is given in the following result.
\begin{lemma}[Proposition 4 in \citet{zhang2020uniqueness}]\label{lem:pseudo_SLLN}
	Suppose $Y_1, Y_2,\ldots$ are independently sampled from $P_{\btheta_0}$ and let $\hat{\btheta}_n=(\hat{\tau}_n,\hat{\mu}_n,\hat{\xi}_n)$, or equivalently $\hat{\bEta}_n=(\hat{\tau}_n,\hat{\mu}_n-\hat{\tau}_n/\hat{\xi}_n,\hat{\xi}_n)$, be the local MLE of $L_n$ that is strongly consistent. Then for constants $k$ and $a$ such that $k\xi_0+a+1>0$,
	\begin{equation*}
		\frac{1}{n}\sumN W_i^{-k-a/\hat{\xi}_n}(\hat{\bEta}_n)\log^b W_i(\hat{\bEta}_n)\rightarrow (-\xi_0)^b\Gamma^{(b)}(k\xi_0+a+1)\quad \text{a.s.},
	\end{equation*}
	where $b$ is a non-negative integer and $\Gamma^{(b)}$ is the $b$th derivative of the Gamma function.
\end{lemma}
Strong consistency of the local MLE is established in \citet{dombry2015existence}. However, care is needed because the boundary of $\Omega_n$, which is confined by $Y_{(1)}$ and $Y_{(n)}$, approaches arbitrarily close to $\hat{\bEta}_n$ (or $\bEta_0$) \cite[Proposition 3]{zhang2020uniqueness}. 
With the limits in Lemma~\ref{lem:pseudo_SLLN}, we can verify the following result.
\begin{corollary}
	Under the assumptions of Lemma~\ref{lem:pseudo_SLLN}, \LZadd{$n^{-1}L_n''(\hat{\bEta}_n)\rightarrow -I({\bEta_0})$} almost surely as $n\rightarrow\infty$, and hence \ref{C1} holds for the GEV likelihood.
\end{corollary}
\begin{remark}
	\LZadd{Given the exact form of the Hessian matrix listed in \citet{zhang2020uniqueness}, it is easy to see that its limit, namely the negative Fisher information matrix $-I({\bEta_0})$, is a function of $\tau_0$ and $\xi_0$ and does not depend on $\beta_0$. Therefore, $I({\bEta_0})$ is well-defined when $\xi_0=0$.}
\end{remark}

Condition~\ref{C2} requires that $L_n''(\bEta)$ be well-approximated by $L_n''(\hat{\bEta}_n)$ when $\bEta$ is near $\hat{\bEta}_n$. The smoothness of the GEV log-likelihood guarantees that \LZadd{the third derivative} $L_n^{'''}(\bEta)$ can be approximated similarly as in Lemma~\ref{lem:pseudo_SLLN} so that \LZadd{the} mean-value theorem applies to $L_n''(\bEta)$. The result is summarized as follows.
\begin{lemma}[Proposition 5 in \citet{zhang2020uniqueness}]\label{lem:hessian}
	Under the assumptions of Lemma~\ref{lem:pseudo_SLLN}, we can choose a small $r>0$ based on $\bEta_0$ to satisfy \ref{C2}. That is, \LZadd{under the Loewner order}, there almost surely exists $N_{\epsilon}$ for any $\epsilon<r$ such that, for any $n>N_{\epsilon}$ and $\bEta\in \Omega_n\cap B_{\epsilon}(\hat{\bEta}_n)$,
	\begin{equation*}
		I-A_0(\epsilon)\leq L''_n(\bEta)\{L''_n(\hat{\bEta}_n)\}^{-1}\leq I+A_0(\epsilon),
	\end{equation*}
	where $A_0(\epsilon)$ is a $3\times 3$ symmetric positive semi-definite matrix whose elements only depend on $\bEta_0$ and the radius $\epsilon$, and whose largest eigenvalue tends to zero as $\epsilon\rightarrow 0$.
\end{lemma}

\subsection{Consistency of the posterior distribution}\label{sec:posteriorConsistency}
The final condition~\ref{C3} requires the posterior probability mass to asymptotically vanish outside neighborhoods of the MLE.  That is,
\begin{equation}\label{eqn:ratio_int_C_n}
	\int_{\Omega_n\setminus B_r(\hat{\bEta}_n)}\pi_n(\bEta)d\bEta=\frac{1}{C_n} \int_{\Omega_n\setminus B_r(\hat{\bEta}_n)}\prod_{i=1}^n p(Y_i\given \bEta)\pi(\bEta)d\bEta\rightarrow 0\quad \text{a.s.}
\end{equation}
Under \ref{C1} and \ref{C2}, the greatest lower bound can be obtained for the normalizing constant $C_n$.
\begin{lemma}\label{lem:lower_bound}
	Under the assumptions of Lemma~\ref{lem:pseudo_SLLN}, let
	\begin{equation}\label{eqn:lower_bound}
		B_n = (2\pi)^{3/2}|-L_n''(\hat{\bEta}_n)|^{-1/2}\pi(\hat{\bEta}_n)\hat{\tau}_n^{-n}\exp(-n)\left\{\prod_{i=1}^n W_i(\hat{\bEta}_n)\right\}^{-1-1/\hat{\xi}_n},
	\end{equation}
	in which $W_i(\bEta)$ is defined in \eqref{wi} and $\hat{\bEta}_n=(\hat{\tau}_n,\hat{\beta}_n,\hat{\xi}_n)$. Then:
	\begin{enumerate}[label=(\roman*)]
		\item There almost surely exists $N$ such that, for any $n>N$, $C_n\geq B_n$.
		\item\label{lem:item2} $\lim_{n\rightarrow\infty}C_n/B_n=1$ almost surely if and only if \ref{C3} holds.
		\item  In addition, $\lim_{n\rightarrow\infty} B_n^{1/n}=\tau_0^{-1} \exp\{-(\xi_0\gamma+\gamma+1)\}$ almost surely, where $\gamma$ is the Euler–Mascheroni constant.
	\end{enumerate}
\end{lemma}
\begin{proof}
	Since \ref{C1} and \ref{C2} have been validated in Section~\ref{sec:cond_1_2}, Lemma 2.1 in \citet{chen1985asymptotic}, i.e., Lemma~\ref{lem:bounded_concentration} in Appendix \ref{appendix:revised_proof}, gives
	\begin{equation*}
		\lim_{n\rightarrow \infty} |-L_n''(\hat{\bEta}_n)|^{-1/2}\pi_n(\hat{\bEta}_n)\leq (2\pi)^{-3/2} 
	\end{equation*}
	almost surely. The equality holds if and only if \ref{C3} holds. Therefore, there almost surely exists $N>0$ such that, for all $n>N$,
	\begin{equation*}
		C_n\geq (2\pi)^{3/2}\lvert -L_n''(\hat{\bEta}_n)\given^{-1/2}\exp\{\log\pi(\hat{\bEta}_n)+L_n(\hat{\bEta}_n)\},
	\end{equation*}
	which can be expanded into the stated lower bound using \eqref{Log-lik} and one of the score equations $\sumN W_i^{-1/\hat{\xi}_n}(\hat{\bEta}_n)=n$. Finally, the limit of the lower bound can be obtained using Lemma~\ref{lem:pseudo_SLLN} and the continuity of the prior around the true parameters $\bEta_0$.
\end{proof}

If \ref{C3} holds, then asymptotic normality is verified for the GEV distribution and $C_n\approx B_n$ for all sufficiently large $n$. Moreover, $C_n\sim \tau_0^{-n} \exp\{-n(\xi_0\gamma+\gamma+1)\}$, which becomes arbitrarily small. To show $\int_{\Omega_n\setminus B_r(\hat{\bEta}_n)}\pi_n(\bEta)d\bEta\rightarrow 0$ almost surely, we study the numerator in \eqref{eqn:ratio_int_C_n} by partitioning the domain of the integral into five regions, so that $\Omega_n\setminus B_r(\hat{\bEta}_n)=\cup _{k=1}^5\Omega_n^k$, and by examining each region separately.  For conciseness, we only verify this condition for the case $\xi_0>0$. The proof for the cases $\xi_0=0$ and $-1/2<\xi_0<0$ are analogous (see Appendix \ref{appendix:outline}). The regions we consider for case $\xi_0>0$ are
\begin{equation*}
	\begin{split}
		\Omega_n^1&=\{\bEta\in \Omega_n\setminus B_{r}(\bEta_0):0<\xi<\xi_0+r_1, \beta_0-r_2<\beta<Y_{(1)}, \tau<\tau_0+r_3\},\\
		\Omega_n^2&=\{\bEta\in \Omega_n:\xi>\xi_0+r_1\},\\ \Omega_n^3&=\{\bEta\in \Omega_n:0<\xi<\xi_0+r_1, \beta<\beta_0-r_2,\tau<\tau_0+r_3\},\\
		\Omega_n^4&=\{\bEta\in \Omega_n:0<\xi<\xi_0+r_1, \tau>\tau_0+r_3\},\; \Omega_n^5=\{\bEta\in \Omega_n:-1/2<\xi<0\},  
	\end{split}
\end{equation*}
where $r_1,\;r_2,\;r_3>r$ are pre-specified constants such that
\begin{equation}\label{eqn:range_split}
	\begin{split}
		\log\left(1+\frac{r_1}{\xi_0}\right)&>\frac{4}{\xi_0}-\log 2+\gamma,\quad r_2>1,\\
		\log\left(1+\frac{r_3}{\tau_0}\right)&>(\xi_0+r_1+1)\log\frac{\xi_0+r_1+1}{e}+(1+\xi_0)\gamma+1.
	\end{split}
\end{equation}
We pre-specify $r_1,\;r_2,$ and $r_3$ by \eqref{eqn:range_split} to make sure that at least one parameter in $\Omega_n^k$, $k=2,3,4$, is sufficiently far away from $\bEta_0$; see Fig.~\ref{fig:range_split} for an illustration. Denote the contributions to the numerator in \eqref{eqn:ratio_int_C_n} corresponding to integrals over these sub-regions as $C_n^{(1)},\ldots,C_n^{(5)}$ respectively. Then condition~\ref{C3} is equivalent to the condition $\lim_{n\rightarrow\infty} C_n^{(k)}/C_n=0$, $k=1,\ldots,5$ almost surely.
\begin{figure}[t]
	\centering
	\includegraphics[width=0.75\linewidth]{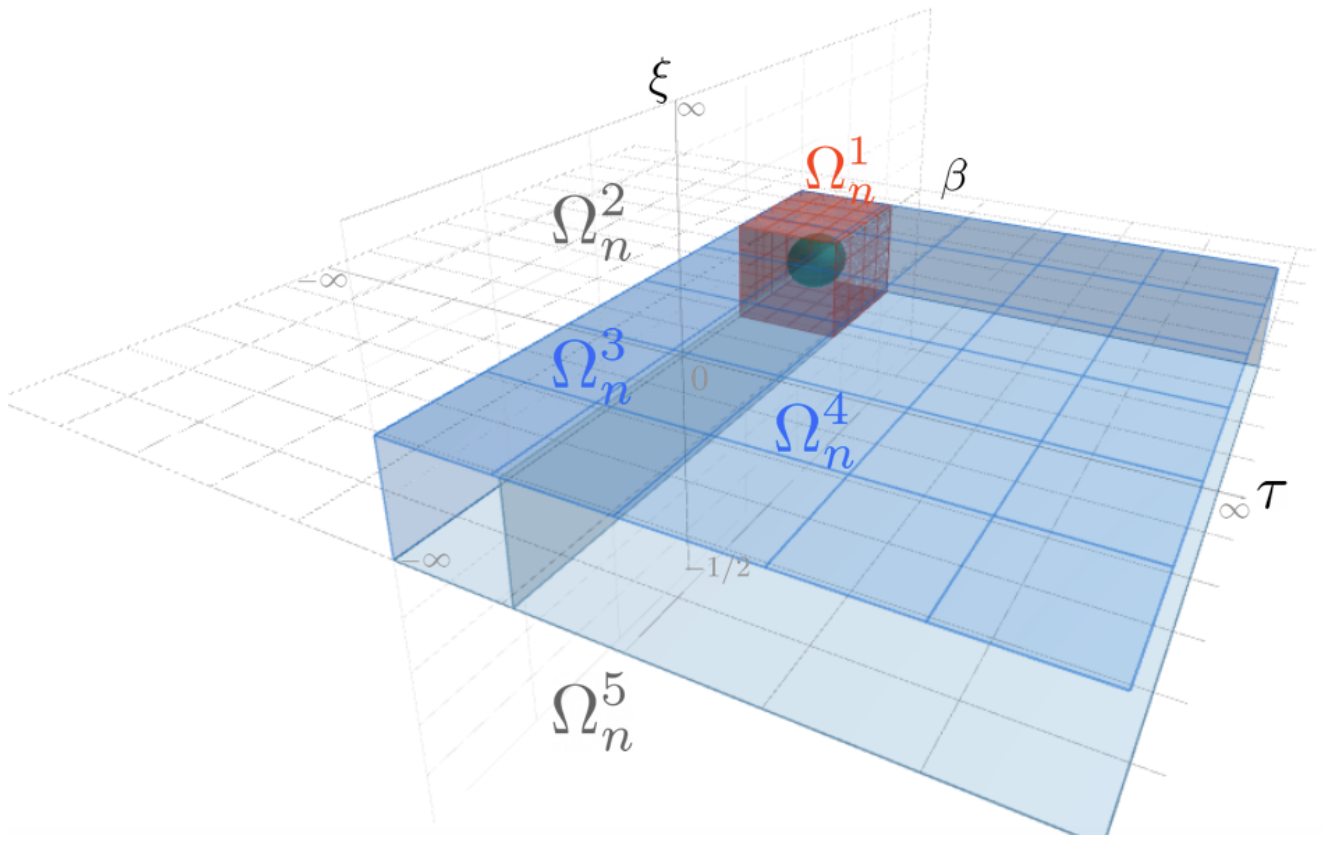}
	\caption{Five sub-regions of $\Omega_n\setminus B_{r}(\hat{\bEta}_n)$ shown under the parametrization $(\tau,\beta,\xi)$ . The neighborhood $B_{r}(\hat{\bEta}_n)$ is shown in dark green. See equation \eqref{eqn:range_split} for details.}
	\label{fig:range_split}
\end{figure}

We begin with $\Omega_n^1$, which envelopes the ball containing $\bEta_0$ and $\hat{\bEta}_0$. To show that it has an asymptotically negligible amount of posterior mass, we apply an open cover argument similar to that of \citet[Proposition 2]{dombry2015existence}.
\begin{proposition}\label{prop:compact_int}
	Let $r>0$ be a small constant, and $K=\{\bEta:0\leq \xi\leq\xi_0+r_1, \beta_0-r_2\leq\beta\leq \beta_0+r_2, 0\leq \tau\leq\tau_0+r_3\}$. 
	Then $K$ is a compact neighborhood of $\bEta_0$. Under the assumptions of Lemma~\ref{lem:pseudo_SLLN},
	\begin{equation}\label{eqn:compact_int}
		\lim_{n\rightarrow\infty}\int_{K\cap\Omega_n\setminus B_{r}(\bEta_0)} \pi_n(\bEta)d\bEta\rightarrow 0 \qquad a.s.
	\end{equation}
\end{proposition}
\begin{proof}
	See Appendix \ref{appendix:C_n_in_K}.
\end{proof}
This result establishes that $\lim_{n\rightarrow\infty} C_n^{(1)}/C_n=0$ almost surely as long as $r_1,\;r_2,\;r_3>r$.  For the other four sub-regions, the integrals can all be iterated one parameter at a time using Fubini's theorem. The calculations are similar to those found in \citet{northrop2016posterior}.  
In contrast to \citet{northrop2016posterior}, however, we derive tighter upper bounds for $C_n^{(k)}$, $k=2,\ldots,5$ in the proof of the following result.
\begin{proposition}\label{prop:integral_all_regions}
	Assume the conditions of Lemma~\ref{lem:pseudo_SLLN}. For the class of prior densities that satisfy Condition \ref{cond_for_prior}, $\lim_{n\rightarrow\infty} C_n^{(k)}/C_n=0$ almost surely, $k=2,\ldots,5$.
\end{proposition}
\begin{proof}
	See Appendix \ref{appendix:C_n_2-5}. The proof relies heavily on the properties of the incomplete gamma and beta functions, which are listed in \citet{olver2010nist}.
\end{proof}
Paired with Proposition~\ref{prop:compact_int}, Proposition~\ref{prop:integral_all_regions} implies that condition~\ref{C3} holds. Since all three conditions in Lemma~\ref{lem:Chen_conditions} have been verified, the main result stated in Theorem~\ref{thm:main_result_posterior_normality} follows.

\section{Simulation study}\label{sec:simulation}
In this section, we perform simulations to examine the sample size $n$ required for the posterior asymptotic normality to manifest.

We simulate i.i.d. samples of sizes $n=50,\; 100,\; 500$, and $1000$ from a GEV distribution with $(\tau_0,\mu_0,\xi_0) = (2, 10, 0.2)$.  Under these parameters, the GEV density $p(y\vert \btheta_0)$ is right-skewed. For each sample, we pair the joint density $\prod_{i=1}^n p(Y_i\given \btheta)$ with the reference prior proposed by \citet{zhang2024reference}
\begin{equation}\label{eqn:h11}
	\pi(\btheta) \propto  \frac{1}{\tau}h^{1/2}_{11}(\xi)
\end{equation}
under the ordered parameterization $(\xi,\mu,\tau)$ or $(\xi,\tau,\mu)$, which prioritizes $\xi$ as the object of interest; see Proposition 2 in \citet{zhang2024reference} for the exact expression of $h_{11}(\xi)$. Then we obtain 100,000 posterior samples from the resulting $\pi_n(\btheta)$ using an adaptive metropolis algorithm \citep{shaby-2010a}  thinning by 10 steps and discarding the first 90,000 samples as burn-in. We then use a kernel density estimator to examine the bivariate marginal posterior densities. The results are shown in Figure \ref{fig:pos_den_est}. We see that the normality starts to emerge when $n=100$ and the joint densities become perfectly bell-shaped as $n$ grows to $500$. This is also supported by the E-statistics test of trivariate normality \citep{rizzo2018statistics}, from which the $p$-values are $2.2\times 10^{-16}$, $0.032$, $0.504$ and $0.604$ for $n=50,\; 100,\; 500$ and $1000$ respectively. Moreover, as $n$ increases, the posterior densities become progressively more concentrated around the MLE $\hat{\btheta}_n$. However, rather than being centered at the true parameters, the posterior densities appear to be centered at the MLEs, which show evident bias for finite samples. The MLE $\hat{\btheta}_n$ is strongly consistent for  $\btheta_0$, but displays non-negligible bias, even when $n=1,000$, especially for $\hat{\xi}_n$. Poor finite-sample properties of estimators of $\xi$ have long been acknowledged in the literature \citep[e.g.]{martins2000generalized}. 
\begin{figure}
	\centering
	\includegraphics[width=\linewidth]{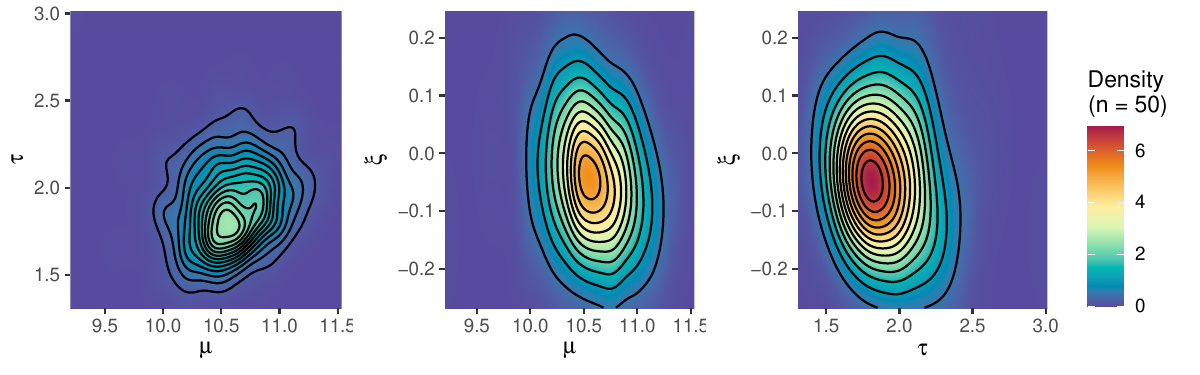}
	\includegraphics[width=\linewidth]{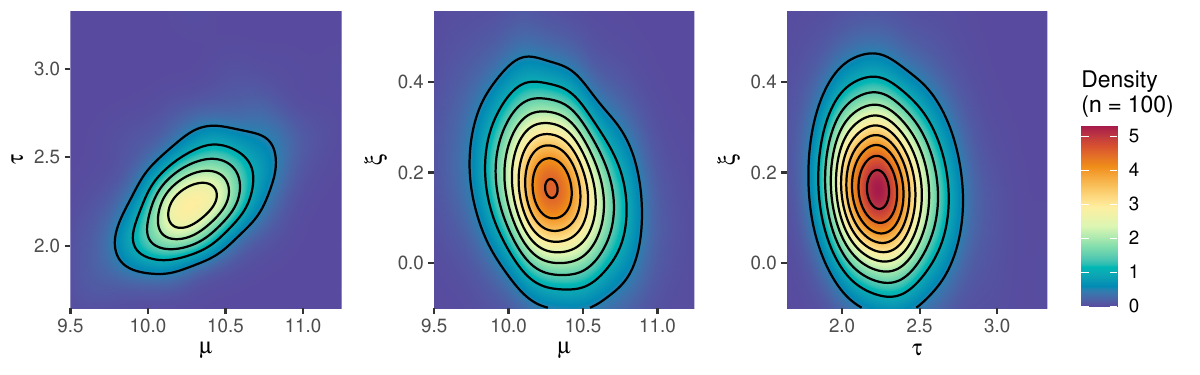}
	\includegraphics[width=\linewidth]{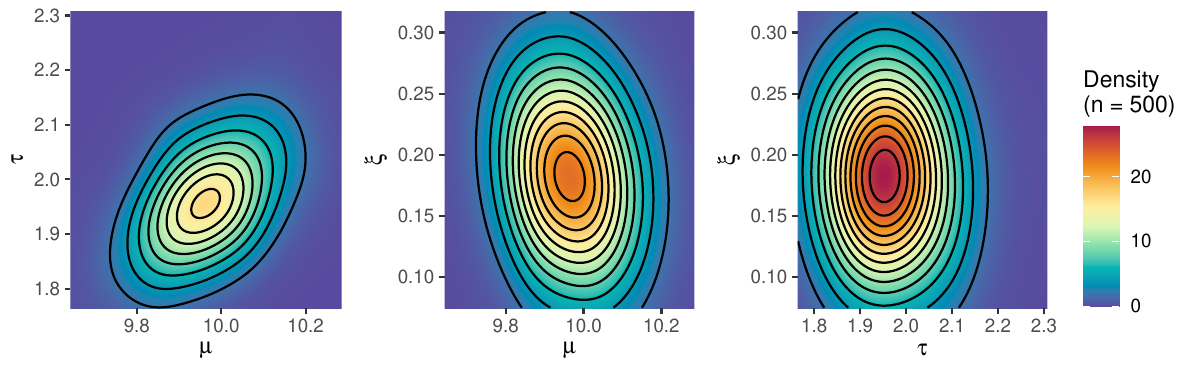}
	\includegraphics[width=\linewidth]{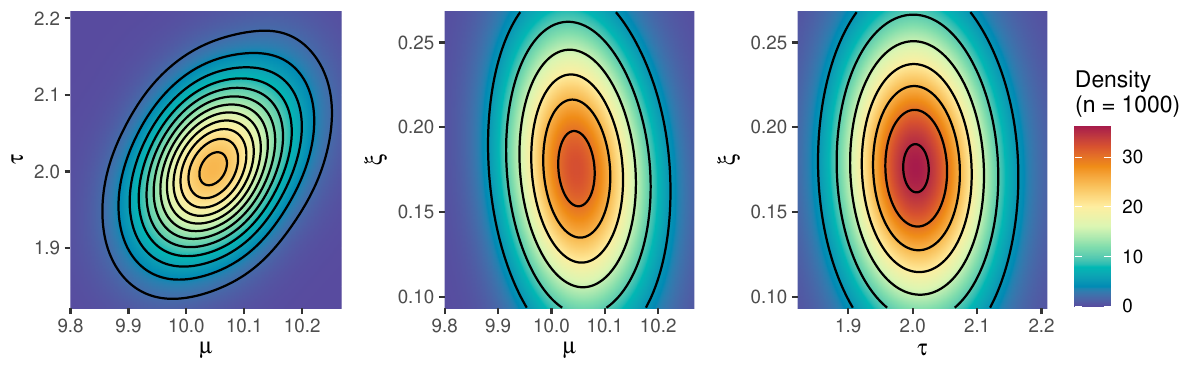}
	\caption{Posterior density estimates  based on i.i.d. samples from the GEV distribution with $(\tau_0,\mu_0,\xi_0) = (2, 10, 0.2)$ and sample sizes $n=50,\;100,\;500,\;1000$. The prior used in this set of simulations is defined in \eqref{eqn:h11}.}
	\label{fig:pos_den_est}
\end{figure}

To further examine the sub-asymptotic biasedness of the MLE, we standardize the posteriors samples of $\btheta$ by the true $\btheta_0$ and $I^{-1}(\btheta_0)/n$, and then by the MLE $\hat{\btheta}_n$ and $\hat{I}^{-1}(\hat{\btheta}_n)/n$, in which $I(\cdot)$ is the Fisher information matrix (see Appendix S1 of \citet{zhang2024reference}) and $\hat{I}(\hat{\btheta}_n)=-n^{-1}L''_n(\hat{\btheta}_n)$. Next, we compare the two standardized samples against the standard normal distribution via QQ-plot. The results for $\xi$ are shown in Figure \ref{fig:qq-plots}, in which we see that the QQ points are highly aligned with the QQ-lines when $n\geq 500$ but the posterior distribution of $\xi$ is centered at $\hat{\xi}_n$ instead of $\xi_0$. See corresponding plots for $\mu$ and $\tau$ in Figure~\ref{fig:qq-plots-tau} and \ref{fig:qq-plots-mu} in Appendix \ref{appendix:simulation}.
\begin{figure}
	\centering
	\includegraphics[width=1.1\linewidth]{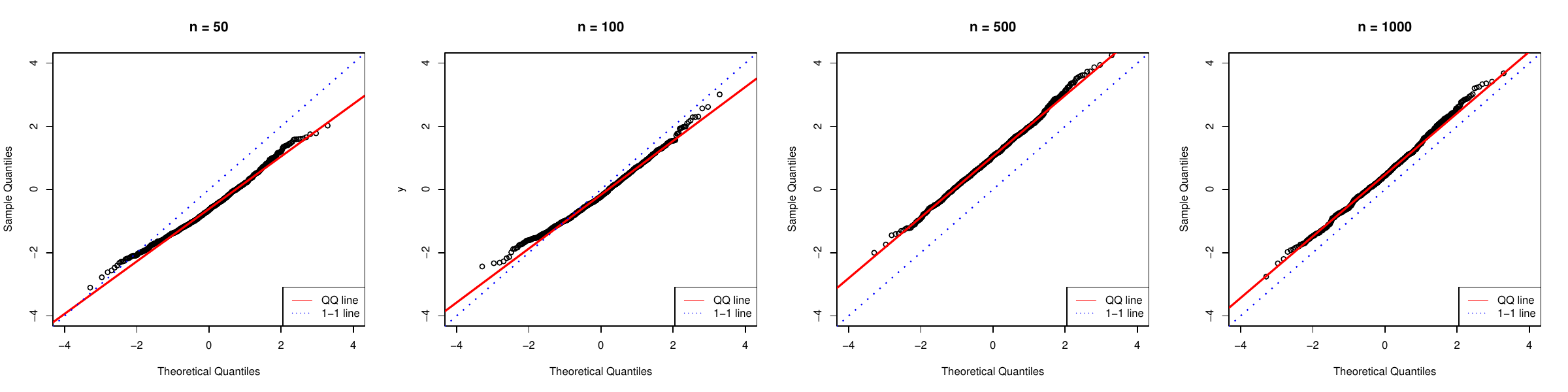}
	\includegraphics[width=1.1\linewidth]{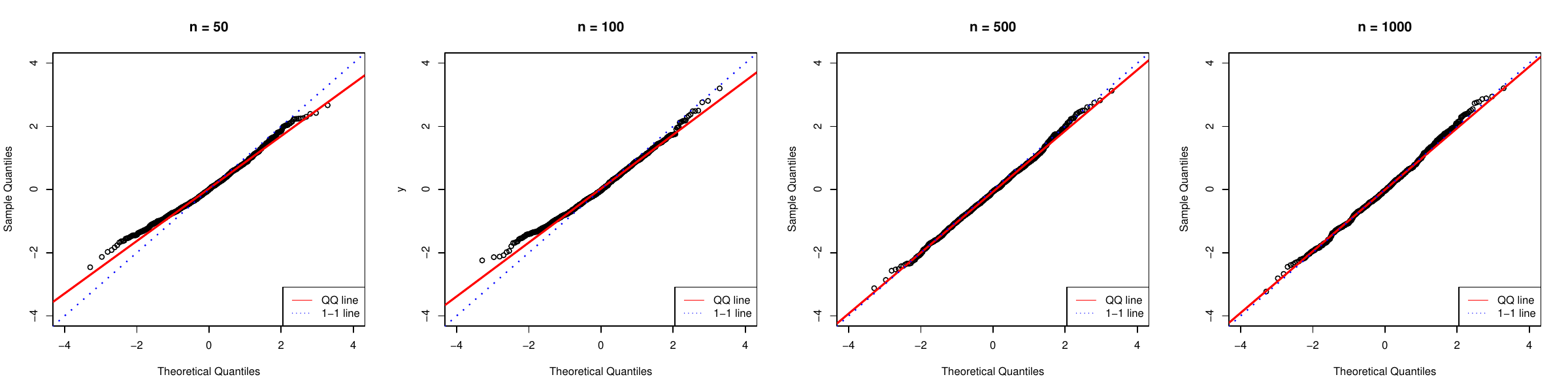}
	\caption{Standard normal QQ-plots for the posterior samples of $\xi$ centered by its asymptotic mean $\xi_0$ and standardized by $(nI_{33})^{-1/2}$ (top panels), and centered by the MLE $\hat{\xi}_n$ and standardized by $(n\hat{I}_{33})^{-1/2}$ (bottom panels). Here, $I_{33}^{-1/2}$ and $\hat{I}_{33}^{-1/2}$ denote the square roots of the third diagonal elements of $I^{-1}(\btheta_0)$ and $\hat{I}^{-1}(\hat{\btheta}_n)$ respectively, where $I(\btheta_0)$ is the Fisher information matrix and $\hat{I}(\hat{\btheta}_n)=-n^{-1}L''_n(\hat{\btheta}_n)$.}
	\label{fig:qq-plots}
\end{figure}

\LZadd{In Appendix \ref{appendix:simulation}, we also show that convergence of the posterior distribution to normality for the GEV family is unaffected by the choice of prior; see Figure \ref{fig:pos_den_est_prior}. We also demonstrate convergence of the posterior distribution under different GEV parameter settings, using $\xi_0=-0.2$ and $\xi_0=0$; see Figures \ref{fig:zero_neg}--\ref{fig:qq_neg} in Appendix \ref{appendix:simulation}.}


\section{Discussion}\label{sec:discussion}
In this paper, we formally derive the asymptotic posterior normality for the family of GEV distributions under a large class of priors. Proposition~\ref{prop:compact_int} demonstrates that the posterior approximation can be easily obtained if the parameter space $\Theta$ is compact, as is generally the case \citep[Chapter 10]{van2000asymptotic}.  The bulk of the technical difficulty lies in the evaluation of the integral of $\prod_{i=1}^n p(Y_i\given \bEta)\pi(\bEta)$ outside the neighborhood of $\bEta_0$, which is closely related to the elliptic integral equivalent to the Carlson $R$-function; see Lemma~\ref{lem:carlson}. An interesting side result from the asymptotic posterior normality is that the lower bound in \eqref{eqn:lower_bound} is attained asymptotically. Therefore, it provides a novel way of estimating an integral that would be difficult to calculate otherwise.

Another major implication from the asymptotic posterior normality is the simplification of the derivation of the reference prior suggested by \citet{berger2009formal}. This class of rule-based priors attempts to maximize the information from the sample so that in a formal sense the prior has a minimal effect on posterior inference. The reference prior formulation is attractive in part because it overcomes the paradoxes of \citet{stein1956inadmissibility} and  \citet{dawid1973marginalization}, which may occur with flat priors and Jeffreys priors. In the general multivariate case, computing reference priors requires multiple applications of a difficult sequence of integrals. Under asymptotic posterior normality, however, the procedure can be streamlined into a rote algorithm  requiring only the Fisher information matrix.

\appendix
\section{Lower bound of the normalizing constant}\label{appendix:revised_proof}
\begin{lemma}\label{lem:bounded_concentration}
Under the assumptions of Lemma~\ref{lem:Chen_conditions}, conditions~\ref{C1} and \ref{C2} together imply that 
\begin{equation}\label{bounded-cctration}
    \lim_{n\rightarrow\infty}\pi_n(\hat{\bphi}_n)|\Sigma_n|^{1/2}\leq (2\pi)^{-k/2}\qquad \text{a.s.},
\end{equation}
where $\Sigma_n=\{-L_n''(\hat{\bphi}_n)\}^{-1}$. The equality holds if and only if \ref{C3} is true.
\end{lemma}
\begin{proof}
Fix $\epsilon>0$, and let $n>\max\{N,N_{\epsilon}\}$ and $r>0$ as given in \ref{C2}. Since $\pi(\bphi)$ is positive and continuous at $\bphi_0$, there exists $0<r_0<r$ and $m(r_0),\; M(r_0)>0$ such that, for any $\bphi_1, \bphi_2\in B_{r_0}(\bphi_0)$, $m(r_0)\leq \pi(\bphi_1)/\pi(\bphi_2)\leq M(r_0)$. Also, both $m(r_0)$ and $M(r_0)$ tend to 1 as $r_0\rightarrow 0$.

By the strong consistency of $\hat{\bphi}_n$,  there almost surely exists $N'>0$ such that for any $n>N'$, $B_{r_0/2}(\hat{\bphi}_n)\subset B_{r_0}(\bphi_0)\subset B_{r}(\bphi_0)$. Then, for $\bphi\in B_{r_0/2}(\hat{\bphi}_n)$, a simple Taylor expansion produces
\begin{equation*}
    \begin{split}
         \pi_n(\bphi)&= \pi_n(\hat{\bphi}_n)\exp\{\log \pi(\bphi)-\log \pi(\hat{\bphi}_n)\}\exp\{L_n(\bphi)-L_n(\hat{\bphi}_n)\}\\
         &=\pi_n(\hat{\bphi}_n)\{\pi(\bphi)/ \pi(\hat{\bphi}_n)\}\exp\left\{-\frac{1}{2}(\bphi-\hat{\bphi}_n)^TL''_n(\bphi_n^+)(\bphi-\hat{\bphi}_n)\right\},
    \end{split}
\end{equation*}
where 
\begin{equation*}
    L''_n(\hat{\bphi}_n)(I-A(\epsilon))\leq L''_n(\bphi_n^+)\leq L''_n(\hat{\bphi}_n)(I+A(\epsilon)),
\end{equation*}
for some $\bphi_n^+$ lying between $\bphi$ and $\hat{\bphi}_n$. Now  consider the posterior probability
\begin{equation*}
    P_n=\int_{B_{r_0/2}(\hat{\bphi}_n)} \pi_n(\bphi)d\bphi.
\end{equation*}
It can be bounded above by
\begin{equation*}
    P_n^+= \pi_n(\hat{\bphi}_n)M(r_0)|\Sigma_n|^{1/2}|I-A(\epsilon)|^{-1/2}\int_{|\boldsymbol{z}|<s_n}\exp(-\boldsymbol{z}^t\boldsymbol{z}/2)d\boldsymbol{z},
\end{equation*}
and bounded below by
\begin{equation*}
    P_n^-= \pi_n(\hat{\bphi}_n)m(r_0)|\Sigma_n|^{1/2}|I+A(\epsilon)|^{-1/2}\int_{|\boldsymbol{z}|<t_n}\exp(-\boldsymbol{z}^t\boldsymbol{z}/2)d\boldsymbol{z},
\end{equation*}
where $s_n=r_0\{1-\underline{a}(\epsilon)\}/(2\underline{\sigma}_n)$ and $t_n=r_0(1+\bar{a}(\epsilon))/(2\bar{\sigma}_n)$, with $\underline{a}(\epsilon)$ and $\underline{\sigma}_n$ being the smallest eigenvalues of $A(\epsilon)$ and $\Sigma_n$, and $\bar{a}(\epsilon)$ and $\bar{\sigma}_n$ being the largest eigenvalues of $A(\epsilon)$ and $\Sigma_n$, respectively.

Because of \ref{C1}, both $s_n$ and $t_n$ tend to $\infty$ as $n\rightarrow\infty$. We can then establish
\begin{equation*}
   \begin{split}
       (2\pi)^{k/2} m(r_0)&|I+A(\epsilon)|^{-1/2}\lim_{n\rightarrow\infty} \pi_n(\hat{\bphi}_n)|\Sigma_n|^{1/2}\leq \lim_{n\rightarrow\infty} P_n\\
       &\leq (2\pi)^{k/2} M(r_0)|I-A(\epsilon)|^{-1/2}\lim_{n\rightarrow\infty} \pi_n(\hat{\bphi}_n)|\Sigma_n|^{1/2}
   \end{split}
\end{equation*}
As $\epsilon\rightarrow 0$, $r_0\rightarrow 0$, and thus $M(r_0)\rightarrow 1$. Moreover, $|I\pm A(\epsilon)|^{1/2}\rightarrow 1$, $m(r_0)\rightarrow 1$ as $\epsilon\rightarrow 0$.
Because $P_n\leq 1$ for all $n$, we can now conclude
\begin{equation*}
    \lim_{n\rightarrow\infty}\pi_n(\hat{\bphi}_n)|\Sigma_n|^{1/2}\leq (2\pi)^{-k/2},
\end{equation*}
where the equality holds if and only if $\lim_{n\rightarrow\infty}P_n=1$.
\end{proof}

\section{Integration over a compact neighbourhood}\label{appendix:C_n_in_K}
Under the $\bEta$-parameterization, denote $l_{\bEta}(y)=\log p(y\given \bEta)$ when $\bEta\in\Theta$ and $l_{\bEta}(y)=-\infty$ when $\bEta\not\in\Theta$, and define $l_B(y)=\sup_{\bEta\in B}l_{\bEta}(y)$ for $B\subset \mathbb{R}^3$. Also, we borrow notation from empirical processes and denote the empirical distribution of $(Y_1,\ldots,Y_n)$ by $\mathbb{P}_n$. Additionally, let $\mathbb{P}_n[f]$ and $P_{\bEta_0}[f]$ signify the integrals of the measurable function $f$ with respect to $\mathbb{P}_n$ and $P_{\bEta_0}$. 

Lemma 6 in \citet{dombry2015existence} guarantees that if $\bEta\in\Theta$ is an interior point of $\Theta$, 
\begin{equation*}
    \lim_{\epsilon\rightarrow 0} P_{\bEta_0}[l_{B(\bEta,\epsilon)}]=P_{\bEta_0}[l_{\bEta}].
\end{equation*}
The following lemma extends this result to $\bEta$ on (part of) the boundary of $\Theta$.

\begin{lemma}\label{lem:dombry_lem6}
If $\xi_0\neq 0$, denote $\Omega=\{\bEta\in\Theta:\text{sgn}(\xi)=\text{sgn}(\xi_0),\xi\{\beta_0-\beta\}>0,\tau>0\}$ and the boundary points of $\Omega$ by $\partial \Omega$. Then for any $\bEta\in \partial \Omega$,
\begin{equation}\label{eqn:dombry_lem6}
    \lim_{\epsilon\rightarrow 0} P_{\bEta_0}[l_{B(\bEta,\epsilon)}]=P_{\bEta_0}[l_{\bEta}],
\end{equation}
in which $P_{\bEta_0}[l_{\bEta}]=-\infty$ if $\bEta\not\in\Theta$.
\end{lemma}
\begin{remark}
The set $\Omega$ arises as the limit of the log-likelihood domain; that is, $\Omega_n\rightarrow\Omega$ almost surely as $n\rightarrow\infty$ (see Appendix E.2 in \citet{zhang2020uniqueness}).
\end{remark}
\begin{proof}
When $\xi_0>0$, $\Omega=\{\bEta:\tau>0, \beta<\beta_0, \xi>0\}$ and $\partial \Omega$ consists of $\{\bEta:\tau=0, \beta<\beta_0, \xi>0\}$, $\{\bEta:\tau>0, \beta=\beta_0,\xi>0\}$ and $\{\bEta:\tau>0, \beta<\beta_0,\xi=0\}$. In the following, we examine each subset of $\partial\Omega$.
\begin{enumerate}[label=(\arabic*)]
    \item $\bEta=(0,\beta,\xi)$ with $\beta<\beta_0$, $\xi>0$. Consider small $\epsilon$ such that $\beta+\epsilon<\beta_0$, $\xi-\epsilon>0$. Recall that the support of $p(y\given\bEta_0)$ is $\{y:y>\beta_0\}$ when $\xi_0>0$. For $y>\beta_0$ and $\bEta'=(\tau',\beta',\xi')\in B(\bEta,\epsilon)\cap\Theta$, 
    \begin{equation}\label{eqn:relaxation}
        \begin{split}
            l_{\bEta'}(y)&=-\log\tau'-\frac{\xi'+1}{\xi'}\log\left\{\frac{\xi'}{\tau'}(y-\beta')\right\}-\left\{\frac{\xi'}{\tau'}(y-\beta')\right\}^{-1/\xi'}\\
            &=\frac{\log\tau'}{\xi'}-\left(1+\frac{1}{\xi'}\right)\log\{\xi'(y-\beta')\}-\left\{\frac{\xi'}{\tau'}(y-\beta')\right\}^{-1/\xi'}\\
            &<\frac{\log\tau'}{\xi'}-\left(1+\frac{1}{\xi'}\right)\log\{(\xi-\epsilon)(y-\beta_0)\}\\
            &<\frac{1}{\xi+\epsilon}\log\tau'+\left(1+\frac{1}{\xi-\epsilon}\right)\left|\log\{(\xi-\epsilon)(y-\beta_0)\}\right|
        \end{split}
    \end{equation}
    and thus
    \begin{small}
        \begin{equation*}
        P_{\bEta_0}[l_{B(\bEta,\epsilon)}]<\frac{1}{\xi+\epsilon}\log\epsilon+\left(1+\frac{1}{\xi-\epsilon}\right)|\log(\xi-\epsilon)|+\left(1+\frac{1}{\xi-\epsilon}\right)E_{\bEta_0}|\log(Y-\beta_0)|.
    \end{equation*}
    \end{small}
    Since $E_{\bEta_0}|\log(Y-\beta_0)|<\infty$ (see Lemma C.1 of \citet{zhang2020uniqueness}), the right-hand side of the last display converges to $-\infty$ as $\epsilon\rightarrow 0$, completing the proof for $\bEta=(0,\beta,\xi)$.
    
    \item $\bEta=(\tau,\beta_0,\xi)$ with $\tau>0$, $\xi>0$. In this case, since $\bEta$ is an interior point of $\Theta$, the result holds automatically by Lemma 6 of \citet{dombry2015existence}.
    
    \item $\bEta=(\tau,\beta,0)$ with $\tau>0$, $\beta<\beta_0$. Consider small $\epsilon$ such that $\tau-\epsilon>0$, $\beta+\epsilon<\beta_0$. Then for $y>\beta_0$ and $\bEta'\in B(\bEta,\epsilon)$ with $\xi'<0$, $l_{\bEta'}(y)=-\infty$ because $\beta'<\beta_0$. Therefore, we only need to consider $\bEta'\in B(\bEta,\epsilon)$ with non-negative $\xi'$. While fixing $\tau'$ and $\beta'$, $l_{\bEta'}(y)$ is an increasing function of $\xi'$ on $(0,\epsilon)$ for small values of $y$. This is due to the fact that $\xi'^{-1/\xi'}$ is the dominating term in $l_{\bEta'}(y)$ when $\epsilon\rightarrow 0$ and the value of $y$ is controlled. Monotonicity holds as long as $y\leq\beta'+\tau'(1+\epsilon)^{-\xi'}/\xi'$, which means
    \begin{equation*}
    l_{\bEta'}(y)\leq  -\log\tau'-\frac{\epsilon+1}{\epsilon}\log\left[\frac{\epsilon}{\tau'}(y-\beta')\right]-\left[\frac{\epsilon}{\tau'}(y-\beta')\right]^{-1/\epsilon}.
    \end{equation*}
    When $y>\beta'+\tau'(1+\epsilon)^{-\xi'}/\xi'$, we have $\xi'(y-\beta')/\tau'>1$ and $l_{\bEta'}(y)<-\log\tau'$. Furthermore, $\beta'+\tau'(1+\epsilon)^{-\xi'}/\xi'>\beta-\epsilon+(\tau-\epsilon)(1+\epsilon)^{-\epsilon}/\epsilon$ for any $\bEta'\in B(\bEta,\epsilon)\cap\{\xi'>0\}$. Thus, we have
    \begin{small}
        \begin{equation*}
        \begin{split}
            l_{B(\bEta,\epsilon)}(y)\leq -\log(\tau+\epsilon)&-\left\{\frac{\epsilon+1}{\epsilon}\log\left[\frac{\epsilon}{\tau+\epsilon}(y-\beta_0)\right]+\left[\frac{\epsilon}{\tau-\epsilon}(y-\beta+\epsilon)\right]^{-1/\epsilon}\right\}\times\\
            &I\{\beta_0<y\leq\beta-\epsilon+(\tau-\epsilon)(1+\epsilon)^{-\epsilon}/\epsilon\}.
        \end{split}
    \end{equation*}
    \end{small}
    
    Denote 
    $C_{\beta,\tau}(\epsilon)=\beta-\epsilon+(\tau-\epsilon)(1+\epsilon)^{-\epsilon}/\epsilon$. As $\epsilon\rightarrow 0$, $C_{\beta,\tau}(\epsilon)\rightarrow\infty$ and
    \begin{equation*}
        E_{\bEta_0}[\log(Y-\beta_0)\cdot I\{\beta_0<Y\leq C_{\beta,\tau}(\xi)\}]\rightarrow E_{\bEta_0}\log(Y-\beta_0)
    \end{equation*}
    and
    \begin{equation*}
        E_{\bEta_0}[(Y-\beta+\epsilon)^{-1/\epsilon}\cdot I\{\beta_0<Y\leq C_{\beta,\tau}(\xi)\}]\sim E_{\bEta_0}(Y-\beta)^{-1/\epsilon}.
    \end{equation*}
    Using this, we deduce that as $\epsilon\rightarrow 0$,
    \begin{small}
        \begin{equation*}
        P_{\bEta_0}[l_{B(\bEta,\epsilon)}]\leq -\log\tau-\frac{1}{\xi}\log\frac{\epsilon}{\tau}-\frac{1}{\xi}E_{\bEta_0}\log(Y-\beta_0)-\left(\frac{\epsilon}{\tau}\right)^{-1/\epsilon}E_{\bEta_0}(Y-\beta)^{-1/\epsilon}\rightarrow -\infty,
    \end{equation*}
    \end{small}
    which completes the proof for this case.
\end{enumerate}
When $\xi_0<0$, $\Omega=\{\bEta:\tau>0, \beta>\beta_0, \xi<0\}$ and $\partial \Omega$ consists of $\{\bEta:\tau=0, \beta>\beta_0, \xi<0\}$, $\{\bEta:\tau>0, \beta=\beta_0,\xi<0\}$ and $\{\bEta:\tau>0, \beta>\beta_0,\xi=0\}$. We can analyze each subset of $\partial\Omega$ similarly as the previous discussion.
\end{proof}

\begin{lemma}\label{lem:fatou}
For all upper semi-continuous functions $f:\mathbb{R}\rightarrow[-\infty,\infty)$ that are upper-bounded,
\begin{equation*}
    \limsup_{n\rightarrow\infty}\mathbb{P}_n[f]\leq P_{\bEta_0}[f]\qquad \text{a.s.}
\end{equation*}
\end{lemma}
\begin{proof}
This is analogous to Lemma 3 in \citet{dombry2015existence}, in which $(Y_i)_{i\geq 1}$ is treated as coming from a distribution in the domain of attraction of a GEV. Since a GEV distribution $P_{\bEta_0}$ is in its own domain of attraction, the block size sequence $m(n)$ is 1 in our setting. Therefore, the condition $\lim_{m(n)\rightarrow \infty}m(n)=+\infty$ is violated.

However, the proof in our simpler setting is similar. Define the empirical distribution function
\begin{equation*}
    \mathbb{F}_n(y)=\frac{1}{n}\sum_{i=1}^n \mathbf{1}_{\{Y_i\leq y\}}, \qquad y\in\mathbb{R},
\end{equation*}
and the corresponding empirical quantile function
\begin{equation*}
    \mathbb{F}^{-1}_n(u):=\inf\{y\in\mathbb{R}:\mathbb{F}_n(y)\geq u\},\qquad u\in (0,1).
\end{equation*}
We can show that $\mathbb{F}^{-1}_n(u)=Y_{(\lceil nu\rceil-1)}$, where $\lceil \cdot\rceil$ is the ceiling function. Since the distribution function $P_{\bEta_0}$ is differentiable at $P_{\bEta_0}^{-1}(u)$ for any $u\in (0,1)$, we have
\begin{equation*}
    Y_{(\lceil nu\rceil-1)}\rightarrow P_{\bEta_0}^{-1}(u)\qquad \text{a.s.}
\end{equation*}

By the upper semi-continuity of $f$,
\begin{equation}\label{semi_cont}
    \limsup_{n\rightarrow\infty} f(\mathbb{F}^{-1}_n(u))\leq f(P_{\bEta_0}^{-1}(u)).
\end{equation}
Meanwhile, we use the relation $\mathbb{P}_n[f]=\int_0^1 f(\mathbb{F}^{-1}_n(u))du$, and $P_{\bEta_0}[f]=\int_0^1 f(P_{\bEta_0}^{-1}(u)) du$. Because $f$ is upper-bounded, we can apply Fatou's lemma to derive
\begin{equation*}
    \limsup_{n\rightarrow\infty}\int_0^1 f(\mathbb{F}^{-1}_n(u))du \leq \int_0^1\limsup_{n\rightarrow\infty} f(\mathbb{F}^{-1}_n(u))du \leq \int_0^1f(P_{\bEta_0}^{-1}(u))du,
\end{equation*}
with the last inequality due to \eqref{semi_cont}. 
\end{proof}

\begin{lemma}\label{lem:posterior_mass_zero_tau}
If $\xi_0\neq 0$, consider $\bEta^*=(0,\beta^*,\xi^*)\in \partial\Omega$ (See Lemma \ref{lem:dombry_lem6} for the definition of $\Omega$). Given prior $\pi(\bEta)=g(\xi)/\tau$, we have 
\begin{equation*}
    \lim_{n\rightarrow\infty}\int_{B(\bEta^*, \epsilon)\cap\Omega_n} \pi_n(\bEta)d\bEta= 0 \qquad\text{a.s.}
\end{equation*}
on condition that the radius $\epsilon$ is sufficiently small such that $\log\epsilon/(4\xi^*)+\log\tau_0+\xi_0\gamma+\gamma+1+c(\xi^*)/2<0$, where $\gamma$ is the Euler–Mascheroni constant and 
\begin{equation*}
   c(\xi^*)=
   \begin{cases}
   2\left(1+\frac{1}{\xi^*}\right)|\log\xi^*|+\left(1+\frac{2}{\xi^*}\right)E_{\bEta_0}|\log(Y-\beta_0)|, &\text{ if }\xi_0>0,\\
   -2\left(1+\frac{1}{\xi^*}\right)|\log(-\xi^*)|-\left(1+\frac{2}{\xi^*}\right)E_{\bEta_0}|\log(\beta_0-Y)|, &\text{ if }\xi_0<0.
   \end{cases}
\end{equation*}
\end{lemma}
\begin{proof}
When $\xi_0>0$, $B(\bEta^*, \epsilon)\cap\partial \Omega\subset\{\bEta:\tau=0, \beta<\beta_0, \xi>0\}$. By \eqref{eqn:relaxation}, we can find sufficiently small $\epsilon>0$ such that for $\bEta\in B(\bEta^*, \epsilon)\cap\Omega_n$,
\begin{equation}\label{eqn:further_relaxation}
    \begin{split}
        P_{\bEta_0}[l_{\bEta}]&\leq \frac{\log\tau}{(\xi^*+\epsilon)}+\left(1+\frac{1}{\xi^*-\epsilon}\right)|\log(\xi^*-\epsilon)|+\left(1+\frac{1}{\xi^*-\epsilon}\right)E_{\bEta_0}|\log(Y-\beta_0)|\\
        &\leq \frac{\log\tau}{(\xi^*+\epsilon)}+2\left(1+\frac{1}{\xi^*}\right)|\log\xi^*|+\left(1+\frac{2}{\xi^*}\right)E_{\bEta_0}|\log(Y-\beta_0)|\\
        &=\frac{\log\tau}{(\xi^*+\epsilon)}+c(\xi^*).
    \end{split}
\end{equation}
Given that $g(\xi)$ is bounded on any interval $[-1/2,c]$ where $c>-1/2$, we can find an upper bound $\Pi_g>0$ for $g(\xi)$ in $[\xi^*-\epsilon, \xi^*+\epsilon]$. Since $B(\bEta^*, \epsilon)\cap\Omega_n\subset \Omega$, the uniform convergence result in \citet[][Equation (E.7) in Appendix E.2]{zhang2020uniqueness} ensures $\lim_{n\rightarrow\infty}n^{-1}L_n(\bEta)= P_{\bEta_0}[l_{\bEta}]$ uniformly for $\bEta\in B(\bEta^*, \epsilon)\cap\Omega_n$. Therefore, there exists $N>0$ such that $L_n(\bEta)<nP_{\bEta_0}[l_{\bEta}]/2$ for all $n>N$ and
\begin{equation*}
    \begin{split}
        \int_{B(\bEta^*, \epsilon)\cap\Omega_n}& \pi(\bEta)\exp\{L_n(\bEta)\}d\bEta\leq\int_{B(\bEta^*, \epsilon)\cap\Omega_n} \pi(\bEta)\exp\{nP_{\bEta_0}[l_{\bEta}]/2\}d\bEta\\
        &\leq \int_{B(\bEta^*, \epsilon)\cap\Omega_n}\pi(\bEta)\exp\left\{\frac{n\log\tau}{2(\xi^*+\epsilon)}+\frac{nc(\xi^*)}{2}\right\}d\bEta\\
        &=\exp\left\{\frac{nc(\xi^*)}{2}\right\}\int_{B(\bEta^*, \epsilon)\cap\Omega_n}g(\xi)\tau^{\frac{n}{2(\xi^*+\epsilon)}-1}d\bEta\\
        &\leq\Pi_g\exp\left\{\frac{nc(\xi^*)}{2}\right\}\cdot\epsilon^{\frac{n}{2(\xi^*+\epsilon)}-1}V(B_\epsilon)\leq\Pi_g\exp\left\{\frac{nc(\xi^*)}{2}\right\}\cdot\epsilon^{\frac{n}{4\xi^*}}V(B_\epsilon),\\
    \end{split}
\end{equation*}
in which $V(B_\epsilon)$ is the volume of $B(\bEta^*, \epsilon)$. 

Furthermore, we know from Lemma \ref{lem:lower_bound} that $C_n\geq \tau_0^{-n}\exp\{-n(\xi_0\gamma+\gamma+1)\}$ almost surely for all sufficiently large $n$. As a consequence,
\begin{equation*}
    \begin{split}
        \int_{B(\bEta^*, \epsilon)\cap\Omega_n} \pi_n(\bEta)d\bEta &\leq \frac{1}{\tau_0^{-n} \exp\{-n(\xi_0\gamma+\gamma+1)\}}\int_{B(\bEta^*, \epsilon)\cap\Omega_n} \pi(\bEta)\exp\{L_n(\bEta)\}d\bEta\\
        &\leq \Pi_gV(B_\epsilon)\exp\left\{n(\log\tau_0+\xi_0\gamma+\gamma+1)+\frac{nc(\xi^*)}{2}+\frac{n}{4\xi^*}\log\epsilon\right\}.
    \end{split}
\end{equation*}
Since $\log\tau_0+\xi_0\gamma+\gamma+1+c(\xi^*)/2+\log\epsilon/(4\xi^*)<0$, we know $\int_{B(\bEta^*, \epsilon)\cap\Omega_n} \pi_n(\bEta)d\bEta\rightarrow 0$ almost surely as $n\rightarrow\infty$.

When $\xi_0<0$, $\partial \Omega=\{\bEta:\tau=0, \beta>\beta_0, \xi<0\}$. For $\bEta\in B(\bEta^*, \epsilon)\cap\Omega_n$ with sufficiently small radius $\epsilon$, we have $-1/2<\xi<0$ and similarly to \eqref{eqn:further_relaxation},
\begin{equation*}
    \begin{split}
        P_{\bEta_0}[l_{\bEta}]&\leq \frac{\log\tau}{(\xi^*+\epsilon)}-\left(1+\frac{1}{\xi^*+\epsilon}\right)|\log(-\xi^*+\epsilon)|-\left(1+\frac{1}{\xi^*+\epsilon}\right)E_{\bEta_0}|\log(\beta_0-Y)|\\
        &\leq \frac{\log\tau}{(\xi^*+\epsilon)}-2\left(1+\frac{1}{\xi^*}\right)|\log(-\xi^*)|-\left(1+\frac{2}{\xi^*}\right)E_{\bEta_0}|\log(\beta_0-Y)|\\
        &=\frac{\log\tau}{(\xi^*+\epsilon)}+c(\xi^*).
    \end{split}
\end{equation*}

\end{proof}

\begin{proof}[Proof of Proposition~\ref{prop:compact_int}]
Denote  $\Delta=K\setminus B_{r_0}(\bEta_0)$. Firstly, we find a finite open cover of the domain of integration $\Delta$ such that the log-likelihood $L_n(\bEta)=n\mathbb{P}_n[l_{\bEta}]$ is nicely bounded for $\bEta$ in each open set of the cover. To achieve this, we recall that the Kullback-Leibler divergence between $P_{\bEta_0}$ and $P_{\bEta}$ is always positive:
\begin{equation*}
    P_{\bEta_0}[l_{\bEta_0}]-P_{\bEta_0}[l_{\bEta}]>0 \text{ if } \bEta\neq \bEta_0.
\end{equation*}
Lemma 6 in \citet{dombry2015existence} and Lemma \ref{lem:dombry_lem6} together ensure that for $\bEta\in\Delta$,
\begin{equation*}
    \lim_{\epsilon\rightarrow 0} P_{\bEta_0}[l_{B(\bEta,\epsilon)}]=P_{\bEta_0}[l_{\bEta}].
\end{equation*}
In particular, for $\bEta\in K\cap\{\xi=0\text{ or } \tau=0\}$, $P_{\bEta_0}[l_{\bEta}]=-\infty$ by Lemma \ref{lem:dombry_lem6}. Thus, for any $\bEta\in\Delta$, we can find $\epsilon_{\bEta}$ such that $P_{\bEta_0}[l_{B(\bEta,\epsilon_{\bEta})}]<-2+P_{\bEta_0}[l_{\bEta_0}]$ if $P_{\bEta_0}[l_{\bEta}]=-\infty$ and 
\begin{equation}\label{eqn:continuity_bound}
    \left|P_{\bEta_0}[l_{B(\bEta,\epsilon_{\bEta})}]-P_{\bEta_0}[l_{\bEta}]\right|<\frac{1}{2}\left(P_{\bEta_0}[l_{\bEta_0}]-P_{\bEta_0}[l_{\bEta}]\right)
\end{equation}
if $P_{\bEta_0}[l_{\bEta}]>-\infty$. Then all the neighborhoods $\{B(\bEta,\epsilon_{\bEta}):\bEta \in \Delta\}$ together constitute an open cover of $\Delta$. Since $\Delta$ is compact, there exists a finite subcover
\begin{equation*}
    \{B(\bEta_i,\epsilon_{\bEta_i}):\;i=1,\ldots,M\},
\end{equation*}
where $M$ is the number of neighborhoods in the subcover. Let $B_i=B(\bEta_i,\epsilon_{\bEta_i})$, $i=1,\ldots, M$. 

Secondly, we prove that there almost surely exists $N>0$ such that for all $n>N$,
\begin{equation}\label{eqn:result}
    \frac{1}{n}L_n(\bEta)-\frac{1}{n}L_n(\hat{\bEta}_n)=\mathbb{P}_n[l_{\bEta}]-\mathbb{P}_n[l_{\hat{\bEta}_n}]\leq c_i,\;\bEta\in B_i,
\end{equation}
where $c_i=(P_{\bEta_0}[l_{\bEta_i}]-P_{\bEta_0}[l_{\bEta_0}])/4$ if $P_{\bEta_0}[l_{\bEta_i}]>-\infty$ and $c_i=-1$ if $P_{\bEta_0}[l_{\bEta_i}]=-\infty$, $i=1,\cdots, M$ . Note that $l_{B_i}$ is an upper semi-continuous function with upper bound. In view of Lemma~\ref{lem:fatou}, we have $\limsup_{n\rightarrow\infty}\mathbb{P}_n[l_{B_i}]\leq P_{\bEta_0}[l_{B_i}]$ almost surely, which implies there almost surely exists $N_i>0$ such that for all $n>N_i$
\begin{equation*}
    \mathbb{P}_n[l_{B_i}]\leq P_{\bEta_0}[l_{B_i}],
\end{equation*}
and thus for any $\bEta\in B_i$,
\begin{equation}\label{eqn:first_appox}
    \begin{split}
        \mathbb{P}_n[l_{\bEta}]-\mathbb{P}_n[l_{\hat{\bEta}_n}]&\leq \mathbb{P}_n[l_{B_i}]-\mathbb{P}_n[l_{\hat{\bEta}_n}]\leq P_{\bEta_0}[l_{B_i}]-\mathbb{P}_n[l_{\hat{\bEta}_n}].\\
    \end{split}
\end{equation}
Meanwhile, according to Lemma~\ref{lem:pseudo_SLLN},
\begin{equation*}
    \mathbb{P}_n[l_{\hat{\bEta}_n}]\rightarrow P_{\bEta_0}[l_{\bEta_0}]\qquad \text{a.s.}
\end{equation*}
Denote $I=\{i=1,\ldots,M: P_{\bEta_0}[l_{\bEta_i}]>-\infty\}$. Then we can almost surely find $N_0>0$ such that for all $n>N_0$, $\mathbb{P}_n[l_{\hat{\bEta}_n}]>P_{\bEta_0}[l_{\bEta_0}]-1$ and
\begin{equation}\label{eqn:LntoL0}
    \left|\mathbb{P}_n[l_{\hat{\bEta}_n}]-P_{\bEta_0}[l_{\bEta_0}]\right|<\min_{i\in I} \frac{1}{4}\left(P_{\bEta_0}[l_{\bEta_0}]-P_{\bEta_0}[l_{\bEta_i}]\right).
\end{equation}
Combining \eqref{eqn:continuity_bound} and \eqref{eqn:LntoL0}, we know for any $n>\max\{N_0,N_1,\ldots,N_M\}$,
\begin{equation*}
    P_{\bEta_0}[l_{B_i}]-\mathbb{P}_n[l_{\hat{\bEta}_n}]< c_i,\;i=1, \ldots, M.
\end{equation*}
Plugging this result back into \eqref{eqn:first_appox} gives \eqref{eqn:result}.

Thirdly, we fix the subcover and split the integral in \eqref{eqn:compact_int} into the integrals over the open sets in $\{B_i:i=1,\ldots,M\}$. Note that constants $c_i,\; i=1,\cdots, M$ are all negative. If $\pi(\bEta)$ is a continuous proper prior density function on $\Theta$, we can find an upper bound $\Pi$ in the compact set $K$. Then
\begin{equation*}
    \begin{split}
      \int_{\Delta\cap\Omega_n} &\pi_n(\bEta)d\bEta\leq \sum_{i=1}^M \int_{B_i\cap\Omega_n} \pi_n(\bEta)d\bEta\\
      &= \pi_n(\hat{\bEta}_n)\sum_{i=1}^M \int_{B_i\cap\Omega_n}\exp\{\log \pi(\bEta)-\log \pi(\hat{\bEta}_n)\}\exp\{L_n(\bEta)-L_n(\hat{\bEta}_n)\}d\bEta\\
      &\leq \Pi^2 \pi_n(\hat{\bEta}_n)\sum_{i=1}^M \int_{B_i\cap\Omega_n} \exp\{L_n(\bEta)-L_n(\hat{\bEta}_n)\}d\bEta,\\
      &\hspace{-0.1cm} \stackrel{\eqref{eqn:result}}{\leq} \Pi^2 \pi_n(\hat{\bEta}_n)\sum_{i=1}^M\exp(c_i n)V(B_i)\stackrel{\eqref{bounded-cctration}}{\leq}\frac{\Pi^2(2\pi)^{-3/2}}{|\Sigma_n|^{1/2}}\sum_{i=1}^M\exp(c_i n)V(B_i)
    \end{split}
\end{equation*}
for any $n>\max\{N_0,\ldots,N_M\}$, 
in which $V(B_i)$ denotes the volume of $B_i$ and is determined by the chosen subcover. Recall that 
\begin{equation*}
    |\Sigma_n|^{-1/2}=O_p(n^{3/2})\qquad \text{a.s.},
\end{equation*}
which yields \eqref{eqn:compact_int}.

If $\pi(\bEta)$ is not bounded in $K$, i.e., $\pi(\bEta)=g(\xi)/\tau$ by Condition \ref{cond_for_prior}, we can still find an upper bound of $\pi(\bEta)$ in the subcover excluding the neighborhoods that intersect with the flat surface $T_0 = \{\tau=0\}$. More specifically, define the index set $\mathcal{I}_0=\{i:B_i\cap T_0= \emptyset, \;i=1,\ldots, M\}$, and there exists an upper bound $\Pi'$ in $\left(\cup_{i\in \mathcal{I}_0}\bar{B}_i\right)\cap \Omega_n$ so we still have  
\begin{equation*}
    \begin{split}
        \sum_{i\in \mathcal{I}_0} \int_{B_i\cap\Omega_n} \pi_n(\bEta)d\bEta \leq \frac{\Pi'^2(2\pi)^{-3/2}}{|\Sigma_n|^{1/2}}\sum_{i\in \mathcal{I}_0}\exp(c_i n)V(B_i).
    \end{split}
\end{equation*}
Now we only need to consider the neighborhoods that intersect with $T_0=\{\tau=0\}$. Suppose $B_j$, $j\in \mathcal{I}_0^c$ is such a neighborhood. Lemma \ref{lem:posterior_mass_zero_tau} asserts that $\int_{B_j\cap\Omega_n} \pi_n(\bEta)d\bEta\rightarrow 0$ using the fact that $L_n(\bEta)$ uniformly tends to zero at a much higher rate than $O(\tau)$ as $\tau\rightarrow 0$. Together we have
\begin{equation*}
    \int_{\Delta\cap\Omega_n}\pi_n(\bEta)d\bEta=\sum_{i\in \mathcal{I}_0} \int_{B_i\cap\Omega_n} \pi_n(\bEta)d\bEta+\sum_{i\in \mathcal{I}^c_0} \int_{B_i\cap\Omega_n} \pi_n(\bEta)d\bEta\rightarrow 0 
\end{equation*}
almost surely as $n\rightarrow\infty$.

\end{proof}


\section{Proof of Proposition~\ref{prop:integral_all_regions}}\label{appendix:C_n_2-5}

\subsection{Proof that \texorpdfstring{$C_n^{(2)}$}{Lg} is negligible}\label{proof:c_n2}
\begin{proof}
Let $Y_{(1)}< \ldots < Y_{(n)}$ be the order statistics. Denote $\delta_j=\delta_j(n)=Y_{(j)}-Y_{(1)}$ and $V_i(\bEta)=\xi(Y_i-\beta)$. Applying the substitutions $t=\{\sum_{i=1}^n V_i^{-1/\xi}(\bEta)\}\tau^{1/\xi}$ and $s=(Y_{(1)}-\beta)^{-1}$ sequentially, we get
\begin{equation}\label{substitutions}
    \begin{split}
        C_n^{(2)}&=\int_{\Omega_n\cap \{\xi>\xi_0+r_1\}}\prod_{i=1}^n p(Y_i\given \bEta)\pi(\bEta)d\bEta\\
        &=\int_{\xi_0+r_1}^{\infty}\int_{-\infty}^{Y_{(1)}}g(\xi)\prod_{i=1}^n V_i^{-1/\xi-1}(\bEta)\int_0^\infty \tau^{n/\xi-1}\exp\left\{-\sum_{i=1}^n W^{-1/\xi}_i(\bEta)\right\}d\tau d\beta d\xi\\
        &=\int_{\xi_0+r_1}^{\infty}\int_{-\infty}^{Y_{(1)}}\frac{g(\xi)\prod_{i=1}^n V_i^{-1/\xi-1}(\bEta)}{\xi^{-1}\{\sum_{i=1}^n V_i^{-1/\xi}(\bEta)\}^n}\int_0^\infty t^{n-1}\exp(-t)dt d\beta d\xi\\
        &=\int_{\xi_0+r_1}^{\infty} g(\xi)\Gamma(n)\xi^{-n+1}\int_{-\infty}^{Y_{(1)}}\frac{\prod_{i=1}^n(Y_i-\beta)^{-{1}/{\xi}-1}}{\left\{\sum_{i=1}^n(Y_i-\beta)^{-1/\xi}\right\}^n} d\beta d\xi\\
        &=\int_{\xi_0+r_1}^{\infty} g(\xi)\Gamma(n)\xi^{-n+1}\int_{0}^{\infty}\frac{s^{n-2}\prod_{j=2}^n(1+\delta_j s)^{-{1}/{\xi}-1}}{\left\{1+\sum_{j=2}^n(1+\delta_j s)^{-1/\xi}\right\}^n} ds d\xi.
    \end{split}
\end{equation}

We first assume $g(\xi)$ is uniform. Since $\left\{1+\sum_{j=2}^n(1+\delta_j s)^{-1/\xi}\right\}^{-1}<1$ for all $s>0$, \eqref{substitutions} can be further relaxed as follows:
\begin{equation*}
    C_n^{(2)}\leq \int_{\xi_0+r_1}^{\infty} \Gamma(n)\xi^{-n+1}\int_{0}^{\infty}\frac{s^{n-2}\prod_{j=2}^n(1+\delta_j s)^{-{1}/{\xi}-1}}{\left\{1+\sum_{j=2}^n(1+\delta_j s)^{-1/\xi}\right\}^{n-1}} ds d\xi.
\end{equation*}
By the inequality of harmonic and geometric means,
\begin{equation}\label{harmonic_means}
    \frac{n}{1+\sum_{j=2}^n(1+\delta_j s)^{-1/\xi}}\leq \left\{\prod_{j=2}^n(1+\delta_j s)^{1/\xi}\right\}^{1/n}.
\end{equation}
Since $n\gg\frac{n}{\log(n-1)}$, there exists $N_1>0$ such that $n-\frac{n}{\log(n-1)}>0$ for any $n>N_1$. Increasing the power of both sides in \eqref{harmonic_means} by $n-1$ further relaxes the lower bound:
\begin{equation*}
    \begin{split}
        C_n^{(2)} &\leq \Gamma(n)n^{-n+1}\int_{\xi_0+r_1}^{\infty} \xi^{-n+1}\int_{0}^{\infty} s^{n-2}\prod_{j=2}^n(1+\delta_j s)^{-1/(n\xi)-1} ds d\xi\\
        &< \frac{\Gamma(n)n^{-n+1}}{\prod_{j=2}^n \delta_j}\int_{\xi_0+r_1}^{\infty} \xi^{-n+1} B\left(n-1, \frac{n-1}{n\xi}\right)d\xi\\
        &<\frac{20\Gamma(n)n^{-n+1}}{\prod_{j=2}^n \delta_j}\sqrt{\frac{n}{n-1}}\int_{\xi_0+r_1}^\infty \xi^{-n+\frac{3}{2}} (n\xi)^{-(n-1)/(n\xi)}d\xi.
    \end{split}
\end{equation*}
The last two inequalities hold due to Lemma~\ref{lem:carlson} and Lemma~\ref{approx_beta} respectively. Furthermore, let $z_n=(n-1)\log\{n(\xi_0+r_1)\}/\{n(\xi_0+r_1)\}$, and we have
\begin{equation}\label{eqn:xi_integral}
    \begin{split}
        \int_{\xi_0+r_1}^\infty \xi^{-n+\frac{3}{2}} &(n\xi)^{-(n-1)/(n\xi)}d\xi \leq \int_{\xi_0+r_1}^\infty \xi^{-n+\frac{3}{2}} \{n(\xi_0+r_1)\}^{-(n-1)/(n\xi)}d\xi\\
        &=\{(\xi_0+r_1)z_n\}^{-n+5/2}\gamma(n-5/2,z_n)\\
        &<\frac{7}{\sqrt{n-5/2}}\{(\xi_0+r_1)z_n\}^{-n+5/2}\exp\{-z_n+(n-5/2)\log z_n\}\\
        &= \frac{7}{\sqrt{n-5/2}}(\xi_0+r_1)^{-n+5/2}\exp\{-z_n\},
    \end{split}
\end{equation}
where the penultimate inequality follows from Lemma~\ref{lem:incomplete_gamma}. Combining the previous two inequalities and the fact that $\Gamma(n)\leq n^{n+\frac{1}{2}}\exp\{-n+1\}$,
\begin{equation*}
\begin{split}
    C_n^{(2)}&\leq \frac{140\Gamma(n)n^{-n+1}(\xi_0+r_1)^{-n+5/2}}{\prod_{j=2}^n \delta_j}\sqrt{\frac{n}{(n-1)(n-5/2)}}\exp\{-z_n\}\\
    &\leq \frac{140(\xi_0+r_1)^{-n+5/2}}{\prod_{j=2}^n \delta_j}\cdot\frac{n^2}{n-5/2}\exp\left\{-z_n-n+1\right\}.
\end{split}
\end{equation*}
Finally, we utilize Lemma \ref{lem:lower_bound} to deduce
\begin{equation}\label{eqn:final_approx}
\begin{split}
    \frac{C_n^{(2)}}{C_n}\leq& M_1\frac{n^{7/2}}{n-5/2}\left(1+\frac{r_1}{\xi_0}\right)^{-n+5/2}\left(\frac{\hat{\xi}_n}{\xi_0}\right)^n\times\left\{(Y_{(1)}-\hat{\beta}_n)\prod_{j=2}^n\frac{Y_{(j)}-\hat{\beta}_n}{\delta_j}\right\}\times\\
    &\left\{\prod_{i=1}^n W^{1/\hat{\xi}_n}_i(\hat{\bEta}_n)\right\}\times\exp\left\{-z_n\right\},
\end{split}
\end{equation}
where 
\begin{equation*} 
    M_1=\frac{280e\xi^{5/2}_0|I(\bEta_0)|^{1/2}}{(2\pi)^{3/2}\pi(\bEta_0)}.
\end{equation*}

Now it suffices to prove that the right-hand side of \eqref{eqn:final_approx} converges almost surely to $0$. Lemma~\ref{lem:delta_diff} gives
\begin{equation*}
   (Y_{(1)}-\hat{\beta}_n)\prod_{j=2}^n\frac{Y_{(j)}-\hat{\beta}_n}{\delta_j}<\frac{\tau_0\exp(4n/\xi_0)}{2^n\xi_0}\qquad \text{a.s.},
\end{equation*}
and Lemma~\ref{lem:pseudo_SLLN} ensures
\begin{equation*}
    \prod_{i=1}^n W^{1/\hat{\xi}_n}_i(\hat{\bEta}_n)=\exp\left\{\frac{1}{\hat{\xi}_n}\sum_{i=1}^n \log W_i(\hat{\bEta}_n)\right\}\sim \exp(n\gamma) \qquad \text{a.s.}
\end{equation*}
The dominating term on the right-hand side of \eqref{eqn:final_approx} is
\begin{equation*}
    \exp\left[-n\left\{\log\left(1+\frac{r_1}{\xi_0}\right)-\log\frac{\hat{\xi}_n}{\xi_0}-\frac{4}{\xi_0}+\log 2-\gamma\right\}\right].
\end{equation*}
Since $\hat{\xi}_n\rightarrow \xi_0$ almost surely and $r_1$ is predetermined to satisfy $\log(1+r_1/\xi_0)>4/\xi_0-\log 2+\gamma$, this quantity must converge almost surely to $0$, which immediately proves that $C_n^{(2)}/C_n\rightarrow 0$ almost surely for when $g(\xi)$ is uniform.

If $g$ is regularly varying at infinity with index $\alpha$, then $g(\xi)=\mathcal{L}(\xi)\xi^{\alpha}$ with $\mathcal{L}(\xi)$ being a slowly varying function at infinity. 
The inequality in \eqref{eqn:xi_integral} can be modified using Cauchy-Schwartz inequality:
\begin{equation}\label{eqn:Cauchy_schwartz}
    \begin{split}
        \int_{\xi_0+r_1}^\infty &g(\xi)\xi^{-n+\frac{3}{2}} (n\xi)^{-(n-1)/(n\xi)}d\xi = \int_{\xi_0+r_1}^\infty \mathcal{L}(\xi)\xi^{-n+\frac{3}{2}+\alpha} (n\xi)^{-(n-1)/(n\xi)}d\xi\\
        &\leq\left\{\int_{\xi_0+r_1}^\infty \mathcal{L}^2(\xi)\xi^{-\LZadd{2}} d\xi\right\}^{1/2}\cdot\left\{\int_{\xi_0+r_1}^\infty \xi^{-2n+\LZadd{5}+2\alpha} (n\xi)^{-2(n-1)/(n\xi)} d\xi\right\}^{1/2}.
    \end{split}
\end{equation}
From the Karamata's Theorem \citep[][p. 17]{resnick2008extreme}, we know the first term in the last line is a finite constant because $\mathcal{L}^2(\xi)\xi^{-\LZadd{2}}\in RV_{-\LZadd{2}}$, while the second term can be relaxed in the same way as \eqref{eqn:xi_integral}:
\begin{footnotesize}
\begin{equation*}
    \begin{split}
        \left\{\int_{\xi_0+r_1}^\infty \right.&\left.\xi^{-2n+\LZadd{5}+2\alpha} (n\xi)^{-2(n-1)/(n\xi)} d\xi\right\}^{1/2}\leq \left[\int_{\xi_0+r_1}^\infty \xi^{-2n+\LZadd{5}+2\alpha} \{n(\xi_0+r_1)\}^{-2(n-1)/(n\xi)} d\xi\right]^{1/2}\\
        &=\{2(\xi_0+r_1)z_n\}^{-n+\LZadd{3}+\alpha}\gamma^{1/2}(2n-\LZadd{6}-2\alpha,2z_n)\\
        &<\{2(\xi_0+r_1)z_n\}^{-n+\LZadd{3}+\alpha}\cdot\frac{\sqrt{7}}{(2n-\LZadd{6}-2\alpha)^{1/4}}\exp\{-z_n+(n-\LZadd{3}-\alpha)\log(2z_n)\}\\
        &=\frac{\sqrt{7}}{(2n-\LZadd{6}-2\alpha)^{1/4}}(\xi_0+r_1)^{-n+\LZadd{3}+\alpha}\exp\{-z_n\},
    \end{split}
\end{equation*}
\end{footnotesize}
where $z_n$ is defined in \eqref{eqn:xi_integral} and the penultimate inequality again follows from Lemma~\ref{lem:incomplete_gamma}. Plug this result back into \eqref{eqn:Cauchy_schwartz}, and we see that the crucial dominating terms in \eqref{eqn:xi_integral} are retained when $g$ is regularly varying at infinity with index $\alpha$. Therefore, the remaining proof can be passed through and we have $C_n^{(2)}/C_n\rightarrow 0$ as $n\rightarrow 0$.

\end{proof}

\begin{lemma}\label{lem:carlson}
Suppose $\delta_j>0$, $j=1,\ldots,n$, and $\min_j \delta_j<\max_j \delta_j$. Then for any $\xi>0$,
\begin{equation}\label{carlson}
    \int_{0}^{\infty} s^{n-1}\prod_{j=1}^n(1+\delta_j s)^{-1-\xi} ds< \frac{B(n,n\xi)}{\prod_{j=1}^n\delta_j}.
\end{equation}
\end{lemma}
\begin{proof}
Integral of the form \eqref{carlson} is an elliptic integral which is equivalent to the Carlson $R$-function, the Dirichlet average of the power series $x^n$. More specifically,
\begin{equation*}
    R_{-n}(\boldsymbol{b},\boldsymbol{\delta})=\frac{1}{B(n,c-n)}\int_{0}^{\infty} s^{n-1}\prod_{j=1}^n(1+\delta_j s)^{-b_i} ds,
\end{equation*}
where $B(\cdot,\cdot)$ is the beta function, $\boldsymbol{b}=(b_1,\ldots, b_n)$, and $c=\sum_{j=1}^n b_j>n$. For the technical details of the Carlson $R$-function and its equivalence to the integrals of the form \eqref{carlson}, see Chapter 6 in \citet{carlson1977special}.

\citet{carlson1966some} provides useful inequalities for the hypergeometric $R$-function. In particular, when $n<c$, 
\begin{equation*}
    R_{-n}(\boldsymbol{b},\boldsymbol{\delta})<K_{-n}(\boldsymbol{b},\boldsymbol{\delta}):=\prod_{j=1}^n \delta_j^{-nb_i/c}.
\end{equation*}
In our case, $b_j=1+\xi$, $c=n+n\xi$ and $nb_i/c=1$, showing that \eqref{carlson} holds.
\end{proof}

\begin{lemma}\label{approx_beta}
Suppose $\{b_n\}_{n\geq 2}$ is a positive sequence that converges to $b>0$. Then there exists $N>0$ such that for any $n>N$ and $\xi>b_n$,
\begin{equation*}
    B\left(n-1, \frac{n-1}{n\xi}\right) < 20\sqrt{\frac{n\xi}{n-1}}(n\xi)^{-\frac{n-1}{n\xi}}.
\end{equation*}
\end{lemma}
\begin{proof}
By Stirling approximation, $\Gamma(x)=\frac{1}{x}\sqrt{2\pi x}\left(\frac{x}{e}\right)^x e^{\theta(x)/(12x)}$, where $\theta(x)\in (0,1)$ and $x>0$. Therefore
\begin{equation*}
     \Gamma\left(\frac{n-1}{n\xi}\right)\leq \sqrt{\frac{2\pi n\xi}{n-1}}\left(\frac{n-1}{n\xi}\right)^{\frac{n-1}{n\xi}}\exp\left\{-\frac{n-1}{n\xi}+1\right\}.
\end{equation*}

Since $\{b_n\}_{n\geq 2}$ has a positive limit, there exists $N>0$ such that $1/(n\xi)<1/(nb_n)<1$ for any $n>N$ and $\xi>b_n$. Apply the Stirling approximation again, and we have
\begin{small}
    \begin{equation*}
\begin{split}
     \frac{\Gamma(n-1)}{\Gamma\left(n-1+\frac{n-1}{n\xi}\right)}&\leq \sqrt{\frac{n-1+\frac{n-1}{n\xi}}{n-1}}\left(\frac{n-1}{n-1+\frac{n-1}{n\xi}}\right)^{n-1+\frac{n-1}{n\xi}}(n-1)^{-\frac{n-1}{n\xi}}\exp\left\{\frac{n-1}{n\xi}+1\right\}\\
     &< \sqrt{2}\left(1-\frac{\frac{n-1}{n\xi}}{n-1+\frac{n-1}{n\xi}}\right)^{n-1+\frac{1}{\xi\log n}}(n-1)^{-\frac{n-1}{n\xi}}\exp\left\{\frac{n-1}{n\xi}+1\right\}\\
     &<\sqrt{2}(n-1)^{-\frac{n-1}{n\xi}}\exp\left\{\frac{n-1}{n\xi}+1\right\}.
\end{split}
\end{equation*}
\end{small}
The last line holds by noticing that $(1-1/x)^x<1$ for any $x>1$.

Since $2e^2\sqrt{\pi}<20$, the desired upper bound can be obtained by combining the previous two inequalities.

\end{proof}

\begin{lemma}\label{lem:incomplete_gamma}
Suppose $\{z_n\}_{n\geq 2}$ and $\{a_n\}_{n\geq 2}$ are \LZadd{two positive sequences that satisfy $z_n=o(n)$ and $a_n=O(n)$ as $n\rightarrow\infty$}. Then $z_n/a_n\rightarrow 0$ as $n\rightarrow\infty$. There exists $N>0$ such that for any $n>N$,
\begin{equation*}
    \gamma\left(a_n,z_n\right)<\frac{7}{\sqrt{a_n}}\exp\{-z_n+a_n\log z_n\},
\end{equation*}
where $\gamma(a,z):=\int_{0}^z t^{a-1}e^{-t}dt$ is the incomplete gamma function.
\end{lemma}
\begin{proof}
We use results from \citet{temme1996uniform} concerning the uniform expansion of the incomplete gamma function $\gamma(a,z)$ as $a\rightarrow\infty$. Denote $\lambda=z/a$ and $\eta^2/2=\lambda-1-\log \lambda$, in which $\text{sgn}(\eta)=\text{sgn}(\lambda-1)$. It follows that
\begin{equation}\label{erfc}
    \frac{\gamma(a,z)}{\Gamma(a)}=\frac{1}{2}\text{erfc}(-\eta\sqrt{a/2})-S_a(\eta),
\end{equation}
where the residual term $S_a(\eta)$ is small and dominated by the error function $\text{erfc}(\cdot)$, which is defined by
\begin{equation*}
    \text{erf} (z) = \frac{2}{\sqrt{\pi}}\int_0^z e^{-t^2}dt,\;\;\text{erfc} (z)=1-\text{erf} z.
\end{equation*}
In our case, $a_n=O(n)$, $a_n\gg z_n$. It is clear that $\lambda_n=z_n/a_n<1$ except for finite terms. By the uniformity of the expansion \eqref{erfc}, there exists $N_1>0$ such that for any $n>N_1$,
\begin{equation}\label{first_incom_gamma}
    \begin{split}
        \frac{\gamma\left(a_n,z_n\right)}{\Gamma(a_n)}&<\text{erfc}(-\eta_n\sqrt{a_n/2})=\text{erfc}(\sqrt{a_n(\lambda_n-1-\log \lambda_n)})\\
        &\leq \exp\left\{-a_n(\lambda_n-1+\log a_n-\log z_n)\right\}.
    \end{split}
\end{equation}
The last inequality utilizes the fact that $\text{erfc} (z)\leq \exp(-z^2)$ for $z>0$.

By Stirling's approximation, 
\begin{equation*}
    \Gamma(a_n)\leq \sqrt{\frac{2\pi}{a_n}}\exp\{a_n(-1+\log a_n)+1\}.
\end{equation*}
Plug this bound into \eqref{first_incom_gamma}, and we get
\begin{equation*}
    \gamma\left(a_n,z_n\right)\leq \sqrt{\frac{2\pi}{a_n}}\exp\left\{-a_n(\lambda_n-\log z_n)+1\right\},
\end{equation*}
which gives the desired upper bound because $\sqrt{2\pi}e<7$.
\end{proof}

\begin{lemma}\label{lem:delta_diff}
Let $Y_1,Y_2,\ldots\stackrel{iid}{\sim}P_{\bEta_0}$, and $\hat{\bEta}_n$ is the local MLE of $L_n(\bEta)$ that is strongly consistent. Then
\begin{equation}\label{eqn:spacing}
   \mathcal{P}\left(\exists N\text{ such that for }n>N,\quad\prod_{j=2}^n \frac{Y_{(j)}-\hat{\beta}_n}{Y_{(j)}-Y_{(1)}}< \frac{\tau_0\exp(4n/\xi_0)}{2^n\xi_0(Y_{(1)}-\hat{\beta}_n)}\right)=1.
\end{equation}
\end{lemma}
\begin{proof}
At first, one may think right away that the product term in the brackets tends to $1$ given that $Y_{(1)}\rightarrow\beta_0$ and $\hat{\beta}_n\rightarrow \beta_0$ almost surely. However, this is not as trivial as it seems because of the different rates of these two convergences and the asymptotic spacings between the order statistics.

\textbf{(I)} First we prove that 
\begin{equation}\label{eqn:unifI}
    \prod_{j=2}^n\frac{\xi_0\left(Y_{(j)}-Y_{(1)}\right)}{\tau_0}\geq\exp\left(n\xi_0\gamma - \frac{4n}{\xi_0}\right)\qquad \text{a.s.}
\end{equation}

Denote $(U_1,\ldots,U_n)=(P_{\bEta_0}(Y_1),\ldots,P_{\bEta_0}(Y_n))$. Then $U_i$, $i=1,\ldots, n$, can be regarded as a sequence of independent uniform random variables. Since
\begin{equation*}
    P_{\bEta_0}(y)=\exp\left[-\left\{\frac{\xi_0}{\tau_0}(y-\beta_0)\right\}^{-1/\xi_0}\right],
\end{equation*}
we can write
\begin{equation}\label{eqn:unif1}
\begin{split}
    Y_{(j)}-Y_{(1)}&=(Y_{(j)}-\beta_0)-(Y_{(1)}-\beta_0)\\
    &=\frac{\tau_0}{\xi_0}\left\{\left(-\log U_{(j)}\right)^{-\xi_0}-\left(-\log U_{(1)}\right)^{-\xi_0}\right\},
\end{split}
\end{equation}
Meanwhile, recall that $|\log x - \log y|\leq |x-y|/\min(x,y)$. We have
\begin{equation}\label{eqn:unif2}
\begin{split}
    \log&\left(-\log U_{(j)}\right)^{-\xi_0}-\log\left\{\left(-\log U_{(j)}\right)^{-\xi_0}-\left(-\log U_{(1)}\right)^{-\xi_0}\right\}\leq \\
    &\frac{\left(-\log U_{(1)}\right)^{-\xi_0}}{\left(-\log U_{(j)}\right)^{-\xi_0}-\left(-\log U_{(1)}\right)^{-\xi_0}}=\frac{1}{\left(\log U_{(j)}/\log U_{(1)}\right)^{-\xi_0}-1}.
\end{split}
\end{equation}
Combine \eqref{eqn:unif1} and \eqref{eqn:unif2} to obtain
\begin{equation}\label{eqn:unif3}
\begin{split}
    \sum_{j=2}^n \log\frac{\xi_0\left(Y_{(j)}-Y_{(1)}\right)}{\tau_0}&=\sum_{j=2}^n \log \left\{\left(-\log U_{(j)}\right)^{-\xi_0}-\left(-\log U_{(1)}\right)^{-\xi_0}\right\}\\
    &\geq \sum_{j=2}^n \log\left(-\log U_{(j)}\right)^{-\xi_0}-\sum_{j=2}^n \frac{1}{\left(\log U_{(j)}/\log U_{(1)}\right)^{-\xi_0}-1}.
\end{split}
\end{equation}

Next we further relax the lower bound in \eqref{eqn:unif3}. On one hand, 
\begin{equation}\label{eqn:unifA}
    \begin{split}
        \frac{1}{n}\sum_{j=2}^n \log\left(-\log U_{(j)}\right)^{-\xi_0}&=\frac{1}{n}\sum_{i=1}^n \log\left(-\log U_i\right)^{-\xi_0}-\frac{\log\left(-\log U_{(1)}\right)^{-\xi_0}}{n}\\
        &\rightarrow E \log(-\log U)^{-\xi_0}-\lim_{n\rightarrow\infty}\frac{\log\left(\log n\right)^{-\xi_0}}{n}=\xi_0\gamma\qquad \text{a.s.}
    \end{split}
\end{equation}
For the second term on the right-hand side of \eqref{eqn:unif3}, we begin with Theorem 2 of \citet{chung1949estimate}, 
\begin{equation*}
    \mathcal{P}\left(\limsup_{n\rightarrow\infty}\frac{n\sup_{t\in\mathbb{R}}|\mathbb{F}_n(t)-F(t)|}{\sqrt{2^{-1}n\log_2n}}=1\right)=1
\end{equation*}
for any continuous distribution function $F$ with $\mathbb{F}_n$ as its empirical distribution function. Equivalently, when $n$ is very large, the classic location of $U_{(j)}$ is at $r/n$, and uniformly
\begin{equation*}
    \left|U_{(j)}-\frac{j}{n}\right|\leq \frac{\log_2n}{2n}, \;j=1,\ldots, n.
\end{equation*}
Therefore, there almost surely exists $N_1>0$ such that for all $n>N_1$,
\begin{equation*}
    1\leq\frac{\log U_{(j)}}{\log U_{(1)}}\leq 1-\frac{\log j}{2\log n},\;j=2,\ldots,n.
\end{equation*}
It follows that
\begin{equation}\label{eqn:unifB}
    \begin{split}
        \sum_{j=2}^n \frac{1}{\left(\log U_{(j)}/\log U_{(1)}\right)^{-\xi_0}-1}&\leq \sum_{j=2}^n \frac{1}{\{1-\log j/(2\log n)\}^{-\xi_0}-1}\\
        &\leq \frac{2\log n}{\xi_0}\sum_{j=2}^n \frac{1}{\log j},
    \end{split}
\end{equation}
where the last inequality holds because $(1-x)^{-\xi_0}-1\geq \xi_0x$ for $x\in (0,1)$ and $\xi_0>0$.

Plugging \eqref{eqn:unifA} and \eqref{eqn:unifB} back into \eqref{eqn:unif3} shows that there almost surely exists $N_2>0$ such that for all $n>N_2$,
\begin{equation*}
    \sum_{j=2}^n \log\frac{\xi_0\left(Y_{(j)}-Y_{(1)}\right)}{\tau_0}
    \geq n\xi_0\gamma - \frac{2\log n}{\xi_0}\sum_{j=2}^n \frac{1}{\log j}.
\end{equation*}
Since $\sum_{j=2}^n 1/\log j\leq \int_{2}^{n+1} (1/\log x) dx=\text{li}(n+1)-\text{li}(2)<2n/\log n$, where $\text{li}(\cdot)$ is the logarithmic integral function, it follows that
\begin{equation*}
    \sum_{j=2}^n \log\frac{\xi_0\left(Y_{(j)}-Y_{(1)}\right)}{\tau_0}\geq n\xi_0\gamma - \frac{4n}{\xi_0},
\end{equation*}
which is exactly \eqref{eqn:unifI} after applying the exponential function on both sides of the previous inequality.

\textbf{(II)} Now we prove \eqref{eqn:spacing}. Lemma~\ref{lem:pseudo_SLLN} ensures
\begin{equation}\label{eqn:delta_comp}
    \left\{\prod_{i=1}^n \frac{\xi_0(Y_i-\hat{\beta}_n)}{\tau_0}\right\}^{1/n}=\exp\left\{\frac{1}{n}\sum_{i=1}^n\log\frac{\xi_0(Y_i-\hat{\beta}_n)}{\tau_0} \right\}\rightarrow \exp(\xi_0\gamma)\qquad \text{a.s.}
\end{equation}
Thus there almost surely exists $N_3>0$ such that for all $n>N_3$,
\begin{equation*}
    \prod_{j=2}^n \frac{\xi_0(Y_{(j)}-\hat{\beta}_n)}{\tau_0} < \left\{\frac{\exp(\xi_0\gamma)}{2}\right\}^n \frac{\tau_0}{\xi_0(Y_{(1)}-\hat{\beta}_n)}.
\end{equation*}

To sum up, we combine the previous inequality with \eqref{eqn:unifI} to conclude that there almost surely exists $N=\max\{N_1, N_2, N_3\}$ such that for all $n>N$,
\begin{equation*}
     \prod_{j=2}^n \frac{Y_{(j)}-\hat{\beta}_n}{Y_{(j)}-Y_{(1)}}< \frac{\tau_0\exp(4n/\xi_0)}{2^n\xi_0(Y_{(1)}-\hat{\beta}_n)},
\end{equation*}
which ends the proof of this lemma.
\end{proof}

\subsection{Proof that \texorpdfstring{$C_n^{(3)}$}{Lg} is negligible}
For $C_n^{(3)}$, we utilize the moments of a generalized Pareto distribution. Suppose $X\sim \text{GP}(\kappa,\tau)$ with density
\begin{equation*}
    f(x)=\frac{1}{\tau}\left(1+\frac{\kappa }{\tau}x\right)^{-1-1/\kappa},\quad x>0.
\end{equation*}
The moment generating function is
\begin{equation*}
    M_X(t)=\sum_{m=0}^\infty \frac{(t\tau)^m}{\prod_{i=0}^m (1-i\kappa)}.
\end{equation*}
Then
\begin{equation}\label{gp_moments}
    EX^m=M_X^{(m)}(0)=\frac{m!\tau^m}{\prod_{i=0}^m (1-i\kappa)},\text{ where }m\kappa<1.
\end{equation}

\begin{proof}[Proof of $C_n^{(3)}$ being negligible]
Fix $r_2>1$. Denote $\Tilde{\beta}=\beta_0-r_2$, $\Tilde{\tau}=\tau_0+r_3$ and $\Tilde{\delta}_i=Y_i-\Tilde{\beta}$. Since $Y_{(1)}>\beta_0$, we know $Y_{(j)}-\Tilde{\beta}>1$ for all $j=1,\ldots,n$. Similarly as in \eqref{substitutions}, we apply substitutions $t=\{\sum_{i=1}^n V_i^{-1/\xi}(\bEta)\}\tau^{1/\xi}$ and $s=(\Tilde{\beta}-\beta)^{-1}$ to get
\begin{equation}\label{eqn:C_n3}
    \begin{split}
        C_n^{(3)}&=\int_{\Omega_n^3}\prod_{i=1}^n p(Y_i\given \bEta)\pi(\bEta)d\bEta\\
        &=\int_{0}^{\xi_0+r_1}\int_{-\infty}^{\beta_0-r_2}g(\xi)\prod_{i=1}^n V_i^{-1/\xi-1}(\bEta)\int_0^{\tau_0+r_3} \tau^{n/\xi-1}\exp\left\{-\sum_{i=1}^n W^{-1/\xi}_i(\bEta)\right\}d\tau d\beta d\xi\\
        &=\int_{0}^{\xi_0+r_1}\int_{-\infty}^{\Tilde{\beta}}\frac{g(\xi)\prod_{i=1}^n V_i^{-1/\xi-1}(\bEta)}{\xi^{-1}\left\{\sum_{i=1}^n V_i^{-1/\xi}(\bEta)\right\}^n}\gamma\left(n,(\tau_0+r_3)^{1/\xi}\sum_{i=1}^n V_i^{-1/\xi}(\bEta)\right) d\beta d\xi\\
         &\leq\int_0^{\xi_0+r_1} \int_{0}^{\infty}g(\xi)\Gamma(n)\xi^{-n+1}\frac{s^{n-2}\prod_{i=1}^n(1+\Tilde{\delta}_i s)^{-{1}/{\xi}-1}}{\left\{\sum_{i=1}^n(1+\Tilde{\delta}_i s)^{-1/\xi}\right\}^n} ds d\xi.\\
         &\leq\int_0^{\xi_0+r_1} \int_{0}^{\infty}g(\xi)\Gamma(n)\xi^{-n+1}\frac{s^{n-2}\prod_{i=1}^{n}(1+\Tilde{\delta}_i s)^{-{1}/{(\xi\log n)}-1}}{n^{n-n/\log n}} ds d\xi.
    \end{split}
\end{equation}
The penultimate line is due to $\gamma(n,\cdot)\leq \Gamma(n)$. The last line is obtained similarly to \eqref{harmonic_means} while raising both sides of the inequality by $n-n/\log n$. 

By \citet[p.130]{dragoslav1964elementary},
\begin{equation*}
    \prod_{i=1}^{n}(1+\Tilde{\delta}_i s)\geq (1+\bar{\delta}_n s)^{n}, \text{ where }\Tilde{\delta}_i>0,\;\prod_{i=1}^{n}\Tilde{\delta}_i=\bar{\delta}_n^{n}.
\end{equation*}
Thus, \eqref{gp_moments} with parameters $(\kappa,\tau)$ such that $1/\kappa=n-1+n/(\xi\log n)$ and $\kappa/\tau=\bar{\delta}_n$ gives
\begin{equation*}
    \begin{split}
        \int_{0}^{\infty}& s^{n-2}\prod_{i=1}^{n}(1+\Tilde{\delta}_i s)^{-{1}/{(\xi\log n)}-1} ds\leq \int_{0}^{\infty} s^{n-2}(1+\bar{\delta}_n s)^{-n-n/(\xi\log n)} ds\\
        &=\frac{\Gamma(n-1)}{(\prod_{i=1}^{n}\Tilde{\delta}_j)^{1-1/n}}\cdot\frac{1}{\prod_{i=0}^{n-2}\{n-1+n/(\xi\log n)-i\}}.
    \end{split}
\end{equation*}
Plug this result back to the right-hand side of \eqref{eqn:C_n3} while letting $M_2>0$ be the upper bound of $g(\xi)$ in $(-1/2,\xi_0+r_1)$ to obtain
\begin{equation*}
    \begin{split}
        C_n^{(3)}&\leq \frac{\Gamma(n-1)\Gamma(n)e^{n}}{n^n(\prod_{i=1}^{n}\Tilde{\delta}_j)^{1-1/n}}\left(\frac{\log n}{n}\right)^{n-1}\int_0^{\xi_0+r_1}\frac{g(\xi)}{\prod_{i=1}^{n-1}(1+n\xi i/\log n)} d\xi\\
        &\leq\frac{\Gamma(n-1)\Gamma(n)e^{n}}{n^n(\prod_{i=1}^{n}\Tilde{\delta}_j)^{1-1/n}}\left(\frac{\log n}{n}\right)^{n-1}\int_0^{\xi_0+r_1}\frac{M_2}{(1+n\Gamma^{1/(n-1)}(n)\xi/\log n)^{n-1}} d\xi\\
        &=\frac{\Gamma(n-1)\Gamma(n)e^{n}}{n^n(\prod_{i=1}^{n}\Tilde{\delta}_j)^{1-1/n}}\left(\frac{\log n}{n}\right)^{n}\frac{M_2}{(n-2)\Gamma^{1/(n-1)}(n)}\times\\
        &\qquad\left[1-\frac{1}{(1+n\Gamma^{1/(n-1)}(n)(\xi_0+r)/\log n)^{n-1}}\right]\\
        &\sim \frac{\Gamma(n-1)\Gamma(n)e^{n}}{n^n(\prod_{i=1}^{n}\Tilde{\delta}_j)^{1-1/n}}\left(\frac{\log n}{n}\right)^{n}\frac{M_2 e}{n(n-2)}.
    \end{split}
\end{equation*}
The third-to-last line is obtained by applying the inequality from \citet[p.130]{dragoslav1964elementary} again, and the last line is due to the Stirling's approximation.

Now we use Lemma \ref{lem:lower_bound} to deduce
\begin{equation}\label{eqn:c_n3_approx}
\begin{split}
    \frac{C_n^{(3)}}{C_n}\leq&M_3\frac{\Gamma(n-1)\Gamma(n)e^{2n+1}\log^n n}{n^{2n+1}(n-2)}\tau_0^{n}\left\{\prod_{i=1}^n W_i(\hat{\bEta}_n)\right\}^{1+1/\hat{\xi}_n}\big/\left(\prod_{i=1}^{n}\Tilde{\delta}_j\right)^{1-1/n},
\end{split}
\end{equation}
where $M_3=(2\pi)^{-3/2}|I(\bEta_0)|^{1/2}\pi^{-1}(\bEta_0)M_2$. Stirling's approximation ensures 
\begin{equation*}
    \frac{\Gamma(n-1)\Gamma(n)e^{2n+1}}{n^{2n+1}(n-2)}\rightarrow 0.
\end{equation*}
Lemma~\ref{lem:pseudo_SLLN} gives
\begin{equation}\label{eqn:w_i_xi_1}
    \prod_{i=1}^n W^{1+1/\hat{\xi}_n}_i(\hat{\bEta}_n)=\exp\left\{\left(1+\frac{1}{\hat{\xi}_n}\right)\sum_{i=1}^n \log W_i(\hat{\bEta}_n)\right\}\sim \exp\{n\gamma(1+\xi_0)\}
\end{equation}
almost surely.
Meanwhile, since $Y_{(j)}-\Tilde{\beta}>1$ for all $j=1,\ldots,n$, $\sum_{i=1}^{n}\log\Tilde{\delta}_j/n\rightarrow\infty$ in contrast to \eqref{eqn:delta_comp}. More specifically,
\begin{equation*}
    \prod_{i=1}^{n}\Tilde{\delta}_j\geq \prod_{j=\lceil \log n\rceil}^{n}\Tilde{\delta}_{(j)}\geq \prod_{j=\lceil \log n\rceil}^{n}(Y_{(j)}-\beta_0)\geq (Y_{(\lceil \log n\rceil)}-\beta_0)^{n-\log n},
\end{equation*}
in which $\delta_{(j)}=Y_{(j)}-\beta_0+r_2$. Similar to \eqref{eqn:unif1} in Lemma~\ref{lem:delta_diff}, we know
\begin{equation*}
   Y_{(\lceil \log n\rceil)}-\beta_0=\frac{\tau_0}{\xi_0}\left\{-\log U_{(\lceil \log n\rceil)}\right\}^{-\xi_0}\geq \frac{\tau_0}{\xi_0}\left(\frac{n}{\log n}\right)^{\xi_0}.
\end{equation*}

Combine the previous results to get 
\begin{equation*}
    \begin{split}
        \tau_0^{n}\log^n n\left\{\prod_{i=1}^n W_i(\hat{\bEta}_n)\right\}^{1+1/\hat{\xi}_n}\big/&\left(\prod_{i=1}^{n}\Tilde{\delta}_j\right)^{1-1/n}\leq\exp[n\{-\xi_0\log n/2-\xi_0\log n(\log\log n)/n+\\
        &\xi_0\log^2 n/n+(1+\xi_0)\log\log n+\log\tau_0+\gamma(1+\xi_0)\}].
    \end{split}
\end{equation*}
This upper bound converges to zero because $-\xi_0\log n/2$ is the dominating term in the curly brackets.

Therefore, the upper bound in \eqref{eqn:c_n3_approx} converges to zero almost surely, which means $C_n^{(3)}/C_n\rightarrow 0$ almost surely.
\end{proof}

\subsection{Proof that \texorpdfstring{$C_n^{(4)}$}{Lg} is negligible}\label{proof:c_n4}
\begin{proof}
Similarly as in \eqref{substitutions}, we obtain 
\begin{equation*}
    \begin{split}
        C_n^{(4)}&=\int_{\Omega_n^4}\prod_{i=1}^n p(Y_i\given \bEta)\pi(\bEta)d\bEta\\
        &=\int_{0}^{\xi_0+r_1}\int_{-\infty}^{Y_{(1)}}g(\xi)\prod_{i=1}^n V_i^{-1/\xi-1}(\bEta)\int_{\tau_0+r_3}^\infty \tau^{n/\xi-1}\exp\left\{-\sum_{i=1}^n W^{-1/\xi}_i(\bEta)\right\}d\tau d\beta d\xi\\
        &=\int_{0}^{\xi_0+r_1}\int_{-\infty}^{Y_{(1)}}\frac{g(\xi)\prod_{i=1}^n V_i^{-1/\xi-1}(\bEta)}{\xi^{-1}\left\{\sum_{i=1}^n V_i^{-1/\xi}(\bEta)\right\}^n}\Gamma\left(n,(\tau_0+r_3)^{1/\xi}\sum_{i=1}^n V_i^{-1/\xi}(\bEta)\right) d\beta d\xi.
    \end{split}
\end{equation*}

By the inequality of harmonic and geometric means,
\begin{equation*}
    \begin{split}
        \prod_{i=1}^n V_i^{-1/\xi-1}(\bEta)&=\left\{\prod_{i=1}^n V_i^{-1/\xi-1}(\bEta)\right\}^{1/n}\times \left\{\prod_{i=1}^n V_i^{-1/\xi}(\bEta)\right\}^{(n-1)(\xi+1)/n}\\
        &\leq \left\{\frac{1}{n}\sum_{i=1}^n V_i^{-1/\xi-1}(\bEta)\right\}\times \left\{\frac{1}{n}\sum_{i=1}^n V_i^{-1/\xi}(\bEta)\right\}^{(n-1)(\xi+1)}.
    \end{split}
\end{equation*}
Plug this inequality back into the right-hand side of the previous display and then apply the substitution $q=(\tau_0+r_3)^{1/\xi}\sum_{i=1}^n V_i^{-1/\xi}(\bEta)=(\tau_0+r_3)^{1/\xi}\sum_{i=1}^n \{\xi(Y_i-\beta)\}^{-1/\xi}$ to obtain
\begin{footnotesize}
\begin{align*}
    C_n^{(4)}&\leq \int_{0}^{\xi_0+r_1}\int_{-\infty}^{Y_{(1)}}\frac{g(\xi)\left\{\sum_{i=1}^n V_i^{-1/\xi}(\bEta)\right\}^{(n-1)\xi-1}}{\xi^{-1}n^{(n-1)\xi+n}}\Gamma\left(n,(\tau_0+r_3)^{1/\xi}\sum_{i=1}^n V_i^{-1/\xi}(\bEta)\right)\times\\
    &\qquad\qquad\sum_{i=1}^n V_i^{-1/\xi-1}(\bEta) d\beta d\xi\\
    &=\int_{0}^{\xi_0+r_1}\frac{g(\xi)(\tau_0+r_3)^{-(n-1)}}{\xi^{-1}n^{(n-1)\xi+n}}\int_{0}^{\infty}q^{(n-1)\xi-1}\Gamma\left(n,q\right) dq d\xi\\
    &=\int_{0}^{\xi_0+r_1}\frac{g(\xi)(\tau_0+r_3)^{-(n-1)}}{n^{(n-1)\xi+n}}\cdot\frac{\Gamma((n-1)\xi+n)}{(n-1)} d\xi.
\end{align*}
\end{footnotesize}
The last equality utilizes the fact that $\int_0^{\infty}q^{a-1}\Gamma(b,q)dq=\Gamma(a+b)/a$ for $a,b>0$. By Stirling's approximation, $\Gamma(x)\leq 3\sqrt{x}(x/e)^{x}$ for large $x$ (e.g., $x>20$). Therefore for $\xi \in (0,\xi_0+r_1)$,
\begin{equation*}
\begin{split}
    \frac{\Gamma((n-1)\xi+n)}{n^{(n-1)\xi+n}}&\leq 3\sqrt{(n-1)\xi+n}\left\{\frac{(n-1)\xi+n}{ne}\right\}^{(n-1)\xi+n}\\
    &< 3\sqrt{n(\xi_0+r_1+1)}\left(\frac{\xi_0+r_1+1}{e}\right)^{(n-1)\xi+n},
\end{split}
\end{equation*}
and it follows that
\begin{equation*}
    \begin{split}
        \int_{0}^{\xi_0+r_1}\frac{g(\xi)\Gamma((n-1)\xi+n)}{n^{(n-1)\xi+n}} d\xi &< 3M_2\sqrt{n(\xi_0+r_1+1)}\int_{0}^{\xi_0+r_1}\left(\frac{\xi_0+r_1+1}{e}\right)^{(n-1)\xi+n}d\xi\\
        &<\frac{3M_2\sqrt{n(\xi_0+r_1+1)}[(\xi_0+r_1+1)/e]^{(n-1)(\xi_0+r_1)+n}}{\log[(\xi_0+r_1+1)/e]},
    \end{split}
\end{equation*}
where $M_2$ is the upper bound of $g(\xi)$ in $(-1/2,\xi_0+r_1)$.

Hence, $C_n^{(4)}$ can be further bounded by 
\begin{equation*}
    C_n^{(4)}\leq M_4\frac{\sqrt{n}}{n-1}\left(\frac{\xi_0+r_1+1}{e}\right)^{(n-1)(\xi_0+r_1+1)}(\tau_0+r_3)^{-n},
\end{equation*}
where 
\begin{equation*}
    M_4=3M_2(\tau_0+r_3)\frac{(\xi_0+r_1+1)^{3/2}/e}{\log[(\xi_0+r_1+1)/e]}.
\end{equation*}
Combining this result with Lemma \ref{lem:lower_bound} and denoting $M_5=2M_4|I(\bEta_0)|^{1/2}/\{(2\pi)^{3/2}\pi(\bEta_0)\}$, we get
\begin{small}
\begin{equation}\label{eqn2:C_n3}
    \frac{C_n^{(4)}}{C_n}\leq \frac{M_5n^2 e^n}{n-1}\left(\frac{\tau_0+r_3}{\hat{\tau}_n}\right)^{-n}\left(\frac{\xi_0+r_1+1}{e}\right)^{(n-1)(\xi_0+r_1+1)}\left\{\prod_{i=1}^n W_i(\hat{\bEta}_n)\right\}^{1+1/\hat{\xi}_n}.
\end{equation}
\end{small}
Similar to the proof in Section~\ref{proof:c_n2}, we now prove the right-hand side of \eqref{eqn2:C_n3} converges almost surely to $0$. 
From \eqref{eqn:w_i_xi_1}, we see that the dominating term on the right-hand side of \eqref{eqn2:C_n3} is 
\begin{equation*}
    \exp\left[n\left\{1+\log \frac{\hat{\tau}_n}{\tau_0}-\log\left(1+\frac{r_3}{\tau_0}\right)+(\xi_0+r_1+1)\log\left(\frac{\xi_0+r_1+1}{e}\right)+(1+\xi_0)\gamma\right\}\right]
\end{equation*}
Since $\hat{\tau}_n\rightarrow \tau_0$ almost surely and $r_3$ is predetermined to satisfy $\log\left(1+r_3/\tau_0\right)>(\xi_0+r_1+1)\log\{(\xi_0+r_1+1)/e\}+(1+\xi_0)\gamma+1$, this dominating term converges almost surely to $0$ and thus $C_n^{(4)}/C_n\rightarrow 0$ almost surely.
\end{proof}

\vspace*{-0.5cm}\subsection{Proof that \texorpdfstring{$C_n^{(5)}$}{Lg} is negligible}\label{proof:c_n5}
\begin{proof}
Denote $\Delta_j=\Delta_j(n)=Y_{(n)}-Y_{(j)}$. Similarly to \eqref{substitutions}, we apply the substitutions $t=\tau^{1/\xi}\sum_{i=1}^n\{\xi(Y_i-\beta)\}^{-1/\xi}$ and $s=(\beta-Y_{(n)})^{-1}$ sequentially:
\begin{equation}\label{eqn:C_n5}
    \begin{split}
        C_n^{(5)}&=\int_{\Omega_n^5}\prod_{i=1}^n p(Y_i\given \bEta)\pi(\bEta)d\bEta\\
        &=\int_{-1/2}^{0}\int_{Y_{(n)}}^{\infty}\frac{g(\xi)\prod_{i=1}^n V_i^{-1/\xi-1}(\bEta)}{(-\xi)^{-1}\left\{\sum_{i=1}^n V_i^{-1/\xi}(\bEta)\right\}^n}\Gamma(n) d\beta d\xi\\
         &=\int_{-1/2}^{0} g(\xi)\Gamma(n)(-\xi)^{-n+1}\int_{0}^{\infty}\frac{s^{n-2}\prod_{j=1}^{n-1}(1+\Delta_j s)^{-{1}/{\xi}-1}}{\left\{1+\sum_{j=1}^{n-1}(1+\Delta_j s)^{-1/\xi}\right\}^n} ds d\xi\\
         &\leq\int_{-1/2}^{0} g(\xi)\Gamma(n)(-\xi)^{-n+1}\int_{0}^{\infty}\frac{s^{n-2}\prod_{j=1}^{n-1}(1+\Delta_j s)^{-{1}/{(2\xi)}-1}}{\left\{1+\sum_{j=1}^{n-1}(1+\Delta_j s)^{-1/\xi}\right\}^n} ds d\xi\\
         &\leq\int_{-1/2}^{0} g(\xi)\Gamma(n)n^{-n}(-\xi)^{-n+1}\int_{0}^{\infty}s^{n-2}\prod_{j=1}^{n-1}(1+\Delta_j s)^{{1}/{(2\xi)}-1} ds d\xi.\\
    \end{split}
\end{equation}
The penultimate line is obtained due to $\xi<0$ and $(1+\Delta_j s)^{1/\xi}<1$, $j=1,\ldots,n-1$. The last line is obtained via applying \eqref{harmonic_means} to the numerator of the inner integrand. By \citet[p.130]{dragoslav1964elementary},
\begin{equation*}
    \prod_{j=1}^{n-1}(1+\Delta_j s)\geq (1+\bar{\Delta}_n s)^{n-1}, \text{ where }\Delta_j>0,\;\prod_{j=1}^{n-1}\Delta_j=\bar{\Delta}_n^{n-1}.
\end{equation*}
Thus, \eqref{gp_moments} with parameters $(\xi,\tau)$ such that $1/\kappa=n-2-(n-1)/(2\xi)$ and $\kappa/\tau=\bar{\Delta}_n$ gives
\begin{equation*}
    \begin{split}
        \int_{0}^{\infty}& s^{n-2}\prod_{j=1}^{n-1}(1+\Delta_j s)^{{1}/{(2\xi)}-1} ds\leq \int_{0}^{\infty} s^{n-2}(1+\bar{\Delta}_n s)^{(n-1)\{{1}/{(2\xi)}-1\}} ds\\
        &=\frac{\Gamma(n-1)}{\prod_{j=1}^{n-1}\Delta_j}\cdot\frac{1}{\prod_{i=0}^{n-2}\{-(n-1)/(2\xi)+n-2-i\}}.
    \end{split}
\end{equation*}
Plug this result back to the right-hand side of \eqref{eqn:C_n5}, and recall that $M_2$ is the upper bound of $g(\xi)$ in $(-1/2,\xi_0+r_1)$:
\begin{equation*}
    \begin{split}
        C_n^{(5)}&\leq \frac{\Gamma(n-1)\Gamma(n)n^{-n}}{\prod_{j=1}^{n-1}\Delta_j}\int_{-1/2}^{0}\frac{g(\xi)}{\prod_{i=0}^{n-2}\{(n-1)/2-(n-2-i)\xi\}} d\xi\\
        &\leq \frac{\Gamma(n-1)\Gamma(n)n^{-n}}{\prod_{j=1}^{n-1}\Delta_j}\left(\frac{2}{n-1}\right)^{n-1}\int_{-1/2}^{0}\frac{M_2}{\{1-2\Gamma^{1/(n-1)}(n-1)\xi/(n-1)\}^{n-1}} d\xi\\
        &=\frac{\Gamma(n-1)\Gamma(n)n^{-n}}{\prod_{j=1}^{n-1}\Delta_j}\cdot\frac{2^{n-1}M_2}{(n-1)^n}\cdot\frac{1-(1+\Gamma^{1/(n-1)}(n-1)/(n-1))^{-n+2}}{2\Gamma^{1/(n-1)}(n-1)/(n-1)}\\
        &\sim \frac{\Gamma(n-1)\Gamma(n)n^{-n}}{\prod_{j=1}^{n-1}\Delta_j}\cdot\frac{2^{n-2}M_2}{(n-1)^n}
    \end{split}
\end{equation*}
The third-to-last line is obtained by applying the inequality from \citet[p.130]{dragoslav1964elementary} again, and the last line is due to the Stirling's approximation.

Once more we use Lemma \ref{lem:lower_bound} to deduce
\begin{equation}
\begin{split}
    \frac{C_n^{(5)}}{C_n}\leq&M_3\frac{\Gamma(n-1)\Gamma(n)2^{n-2}e^n}{n^n(n-1)^n}\tau_0^{n}\left\{\prod_{i=1}^n W_i(\hat{\bEta}_n)\right\}^{1+1/\hat{\xi}_n}\big/\prod_{j=1}^{n-1}\Delta_j\rightarrow 0\qquad \text{a.s.},
\end{split}
\end{equation}
where $M_3=(2\pi)^{-3/2}|I(\bEta_0)|^{1/2}\pi^{-1}(\bEta_0)M_2$.
\end{proof}


\section{Outline of proof for negative and zero shapes}\label{appendix:outline}
\subsection{The case \texorpdfstring{$\xi_0\in (-1/2,0)$}{Lg}}
The regions we consider when $\xi_0\in (-1/2,0)$ are
\begin{equation*}
    \begin{split}
        \Omega_n^1&=\{\bEta\in \Omega_n\setminus B_{r}(\bEta_0):\xi_0-r_1<\xi<0, Y_{(n)}<\beta<\beta_0+r_2, \tau<\tau_0+r_3\},\\
      \Omega_n^2&=\{\bEta\in \Omega_n:-1/2<\xi<\xi_0-r_1\},\\
      \Omega_n^3&=\{\bEta\in \Omega_n:\xi_0-r_1<\xi<0, \beta>\beta_0+r_2,\tau<\tau_0+r_3\},\\
      \Omega_n^4&=\{\bEta\in \Omega_n:\xi_0-r_1<\xi<0, \tau>\tau_0+r_3\},\; \Omega_n^5=\{\bEta\in \Omega_n:\xi>0\},  
      \end{split}
\end{equation*}
and $r_1,\;r_2,\;r_3>r$ are pre-specified constants such that
\begin{equation*}
    \begin{split}
        \log\left(1-\frac{r_1}{\xi_0}\right)&>-\frac{4}{\xi_0}-\log 2+\gamma,\quad r_2>1,\\
        \log\left(1+\frac{r_3}{\tau_0}\right)&>(r_1+1-\xi_0)\log\frac{r_1+1-\xi_0}{e}+(1-\xi_0)\gamma+1.
    \end{split}
\end{equation*}
Again, we denote the contributions to the numerator in \eqref{eqn:ratio_int_C_n} from integrals over these sub-regions as $C_n^{(1)},\ldots,C_n^{(5)}$. Because Lemma \ref{lem:dombry_lem6}--\ref{lem:posterior_mass_zero_tau} hold for $\xi_0\neq 0$, the proof for $\lim_{n\rightarrow\infty} C_n^{(1)}/C_n=0$ (i.e. Proposition \ref{prop:compact_int}) when $\xi_0\in (-1/2,0)$ remains the same as in Appendix \ref{appendix:C_n_in_K}. To show $\lim_{n\rightarrow\infty} C_n^{(k)}/C_n=0$, $k=2,\ldots, 5$ almost surely, we simply need to make fine adjustments to the proofs in Appendix \ref{proof:c_n2}--\ref{proof:c_n5} respectively.

\subsection{The case \texorpdfstring{$\xi_0=0$}{Lg}}\label{appdx:zero_shape}
When $\xi_0=0$, $\hat{\xi}_n=O_p(n^{-1/2})$ is arbitrarily small and $\hat{\xi}_n$ can be either positive or negative,
which greatly complicates the behavior of $L_n$ when $\xi$ is around $0$. However, when $\xi_0=0$, the log-likelihood $L_n$ is highly peaked in the rectangular box $\{|\tau-\tau_0|<1, |\mu-\mu_0|<1, \xi\in [-\exp(\gamma)/\log n, \exp(\gamma)/ \log \log n]\}$ (note that the width of the range of $\xi$ is arbitrarily small for large $n$) \citep[Appendix F]{zhang2020uniqueness}. \LZadd{Therefore it is easier to work with the $\btheta$-parameterization}.  We can show by following the proof in Appendix \ref{proof:c_n2} and \ref{proof:c_n5} that the contributions to the numerator in \eqref{eqn:ratio_int_C_n} from integrals over
\begin{equation*}
    \begin{split}
      \Omega_n^2&=\{\btheta\in \Omega_n:\xi>\xi_0+r_1=r_1\},\;  \Omega_n^5=\{\btheta\in \Omega_n:-1/2<\xi<\xi_0-r_1=-r_1\},  
      \end{split}
\end{equation*}
are negligible for any fixed constant $r_1\in(0,1/2)$.

In the meantime, we can dissect $\Omega_n\cap \{\xi\in [-r_1,r_1]\}$ as follows:
\begin{equation*}
    \begin{split}
        \Omega_n^1&=\{\btheta\in \Omega_n\setminus B_{r}(\btheta_0):-r_1<\xi<r_1, \mu_0-r_2<\mu<\mu_0+r_2, 0<\tau<\tau_0+1\},\\
       \Omega_n^3&=\{\btheta\in \Omega_n:-r_1<\xi<r_1, |\mu-\mu_0|>r_2,\tau<\tau_0+1\},\\
      \Omega_n^4&=\{\btheta\in \Omega_n:-r_1<\xi<r_1, \tau>\tau_0+1\},
      \end{split}
\end{equation*}
where $r_2>\max\{\tau_0\gamma+\zeta, 2(\tau_0+1)\}$ with $\zeta> \tau_0\gamma+\gamma+1$. To prove $\lim_{n\rightarrow\infty} C_n^{(1)}/C_n=0$, we show that $\lim_{\epsilon\rightarrow 0} P_{\btheta_0}[l_{B(\btheta,\epsilon)}]=-\infty$ \LZadd{and $\lim_{n\rightarrow\infty}\int_{B(\btheta, \epsilon)\cap\Omega_n}\pi_n(\btheta)d\btheta = 0$} for $\btheta$ on the boundary of $\Omega_n^1$ with $\tau=0$, similarly to Lemma \ref{lem:dombry_lem6} \LZadd{and \ref{lem:posterior_mass_zero_tau}}. Combining this result and Lemma 6 in \citet{dombry2015existence} again establishes $\lim_{n\rightarrow\infty} C_n^{(1)}/C_n=0$ in the same way as the proof for Proposition \ref{prop:compact_int}. To show $\lim_{n\rightarrow\infty} C_n^{(4)}/C_n=0$, one can slightly alter the proof in Appendix \ref{proof:c_n4}. \LZadd{Lastly, the proof of $\lim_{n\rightarrow\infty} C_n^{(3)}/C_n=0$ is given in Lemma \ref{lem:cn_3_zero_shape}.}

\LZadd{\begin{lemma}\label{lem:cn_3_zero_shape}
Suppose $Y_1, Y_2,\ldots$ are independently sampled from $P_{\btheta_0}$ with $\xi_0=0$. Then $\lim_{n\rightarrow\infty} \int_{\Omega_n^3}\pi_n(\btheta) d\btheta=0$ almost surely.
\end{lemma}}
\begin{proof}
\LZadd{When $-1/2<-r_1<\xi<r_1$, note that $(1+a\xi)^{1/\xi}$ is a decreasing function of $\xi$ for any fixed $a\in \mathbb{R}$ for which in $1+a\xi>0$. 
Then we have
\begin{equation}\label{eqn:jensen0}
    \begin{split}
        \frac{1}{n}\sumN W_i^{-1/\xi}(\btheta) &= \frac{1}{n}\sumN \left\{1+\frac{\xi(Y_i-\mu)}{\tau}\right\}^{-1/\xi}\geq \frac{1}{n}\sumN\left\{1-\frac{Y_i-\mu}{2\tau}\right\}^2\\
        &=1-\frac{\bar{Y}_n-\mu}{\tau}+\frac{1}{n}\sumN\frac{(Y_i-\mu)^2}{4\tau^2}\geq 1-\frac{\bar{Y}_n-\mu}{\tau}+\frac{(\bar{Y}_n-\mu)^2}{4\tau^2},
    \end{split}
\end{equation} 
in which $\bar{Y}_n$ is the sample mean which tends to $\mu_0+\tau_0\gamma$ almost surely as $n\rightarrow\infty$, and the last inequality uses the convexity of the quadratic function. }

\LZadd{Meanwhile, $r_2>\max\{\tau_0\gamma+\zeta, 2(\tau_0+1)\}$ ensures that $(\bar{Y}_n-\mu)/\tau\in (-\infty,0)\cup (2, \infty)$ for all large $n$ and all $\btheta\in \Omega_n^3$. The concavity of the logarithm function gives
\begin{equation}\label{eqn:jensen2}
   \begin{split}
    n^{-1}&\left(-1-\frac{1}{\xi}\right)\sumN \log W_i(\btheta)\leq 
    \left(-1-\frac{1}{\xi}\right) \log \left\{1+\frac{\xi(\bar{Y}_n-\mu)}{\tau}\right\}\\
    &\leq \lim_{\xi\nearrow 0}\left(-1-\frac{1}{\xi}\right) \log \left\{1+\frac{\xi(\bar{Y}_n-\mu)}{\tau}\right\} = -\frac{\bar{Y}_n-\mu}{\tau},
   \end{split}
\end{equation}
in which we also utilized the fact that $(-1-1/\xi)\log(1+\xi a)$ is an increasing function of $\xi$ in $\{\xi:1+a\xi>0, \;a\in (-\infty,0)\cup (2, \infty)\}$. }

\LZadd{Combining \eqref{eqn:jensen0} and \eqref{eqn:jensen2}, we have 
\begin{equation*}
    \begin{split}
        n^{-1}L_n(\btheta) &= -\log\tau+n^{-1}\left(-1-\frac{1}{\xi}\right)\sumN \log W_i(\btheta)-n^{-1}\sumN W_i^{-1/\xi}(\btheta)\\
        &=-\log\tau-1-\frac{(\bar{Y}_n-\mu)^2}{4\tau^2}<-\log\tau-1-\frac{\zeta^2}{4\tau^2}.
    \end{split}
\end{equation*}
The last inequality of the last display again exploits the fact that $r_2>\max\{\tau_0\gamma+\zeta, 2(\tau_0+1)\}$ and $\bar{Y}_n$ tends to $\mu_0+\tau_0\gamma$ almost surely. Consequently,
\begin{equation*}
    \begin{split}
        \int_{\Omega_n^3}  \pi(\btheta)\exp\{L_n(\btheta)\}d\btheta&<\int_{\Omega_n^3}  g(\xi)\tau^{-n-1}\exp\left\{-\frac{n\zeta^2}{4\tau^2}-n\right\}d\btheta\\
        &<\frac{\Pi_g}{2}\left(\frac{4}{n\zeta}\right)^{(n-1)/2}\Gamma\left(\frac{n}{2}\right)e^{-n},
    \end{split}
\end{equation*}
where $\Pi_g$ again denotes the upper bound of $g(\xi)$ in $\xi\in (-r_1, r_1)$. Similarly to the proof of Lemma \ref{lem:posterior_mass_zero_tau}, we can easily show that $$\int_{\Omega_n^3}\pi_n(\btheta) d\btheta = \int_{\Omega_n^3}  \pi(\btheta)\exp\{L_n(\btheta)\}d\btheta/C_n\rightarrow 0$$
as $n\rightarrow\infty$ almost surely on account of Lemma \ref{lem:lower_bound} and the fact that $\zeta> \tau_0\gamma+\gamma+1$.}

\end{proof}

\section{More simulation results}\label{appendix:simulation}
\begin{figure}
    \centering
    \includegraphics[width=1.1\linewidth]{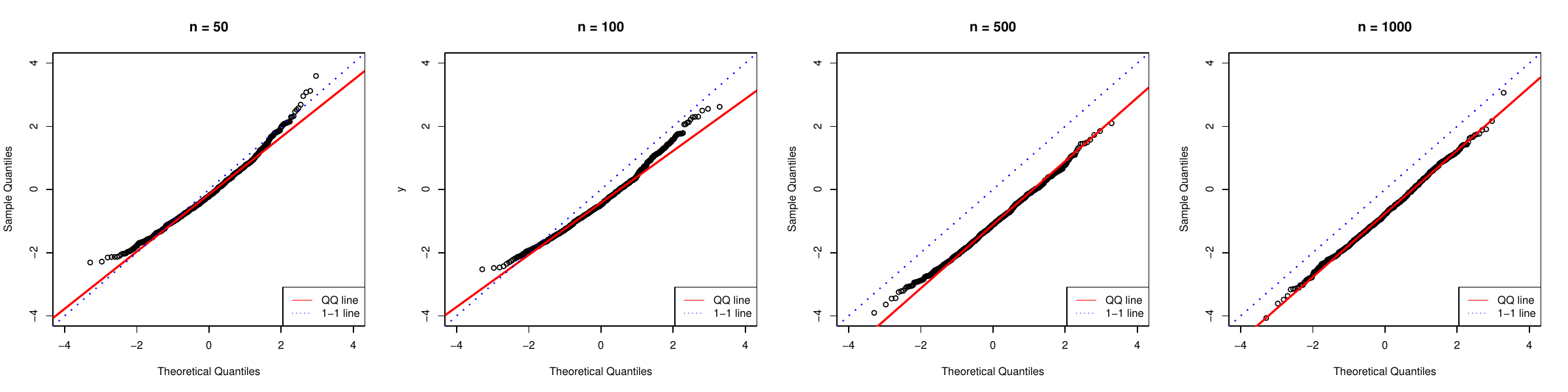}
    \includegraphics[width=1.1\linewidth]{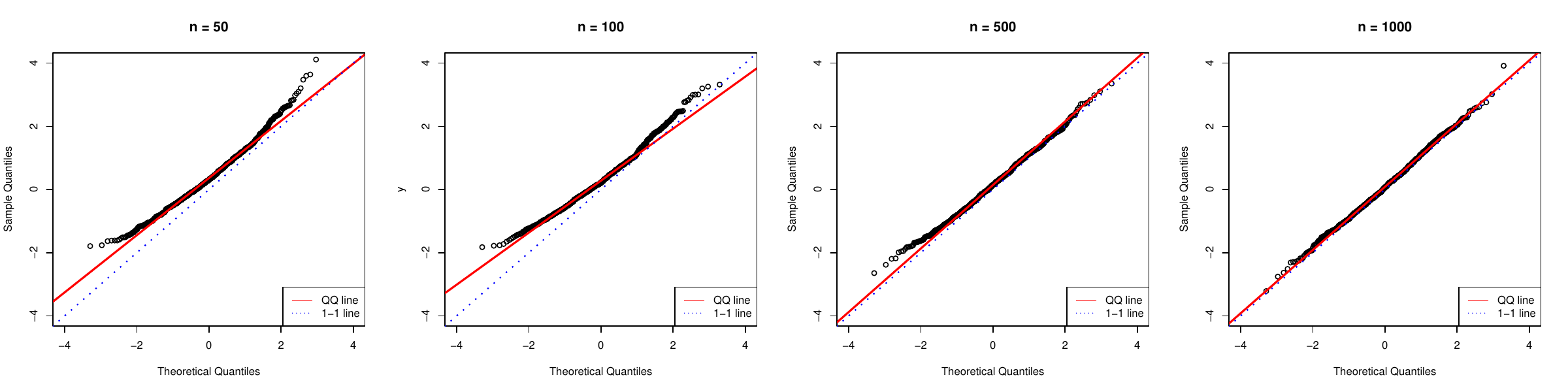}
    \caption{Under the settings of Figure \ref{fig:pos_den_est}, we complement Figure \ref{fig:qq-plots} by showing the QQ-plots for the posterior samples of $\tau$ centered by its asymptotic mean $\tau_0$ and standardized by $(nI_{11})^{-1/2}$ (top panels), and centered by the MLE $\hat{\tau}_n$ and standardized by $(n\hat{I}_{11})^{-1/2}$ (bottom panels). Here, $I_{11}^{-1/2}$ and $\hat{I}_{11}^{-1/2}$ denote the square roots of the first diagonal elements of $I^{-1}(\btheta_0)$ and $\hat{I}^{-1}(\hat{\btheta}_n)$ respectively.}
    \label{fig:qq-plots-tau}
\end{figure}

\begin{figure}
    \centering
    \includegraphics[width=1.1\linewidth]{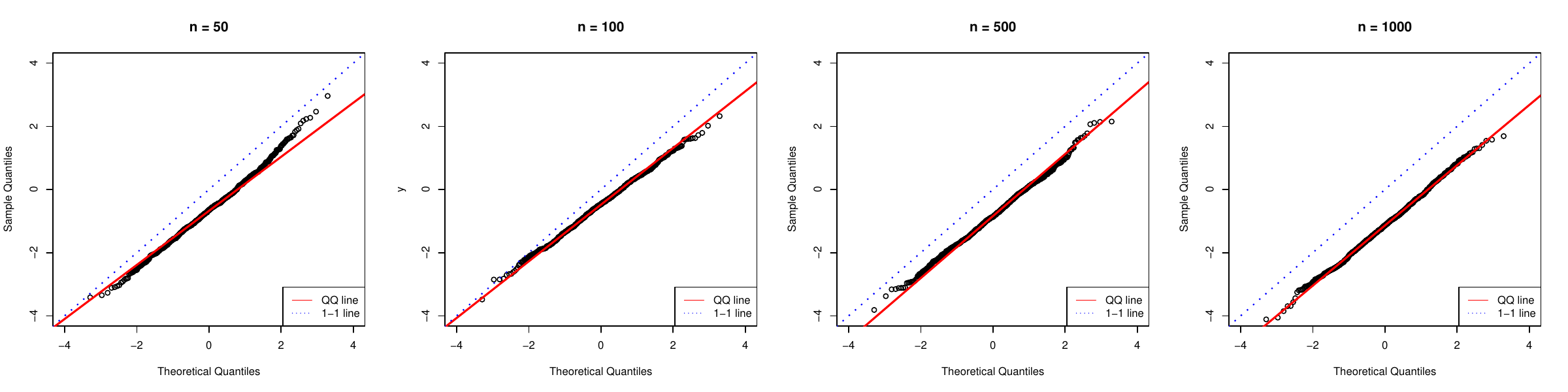}
     \includegraphics[width=1.1\linewidth]{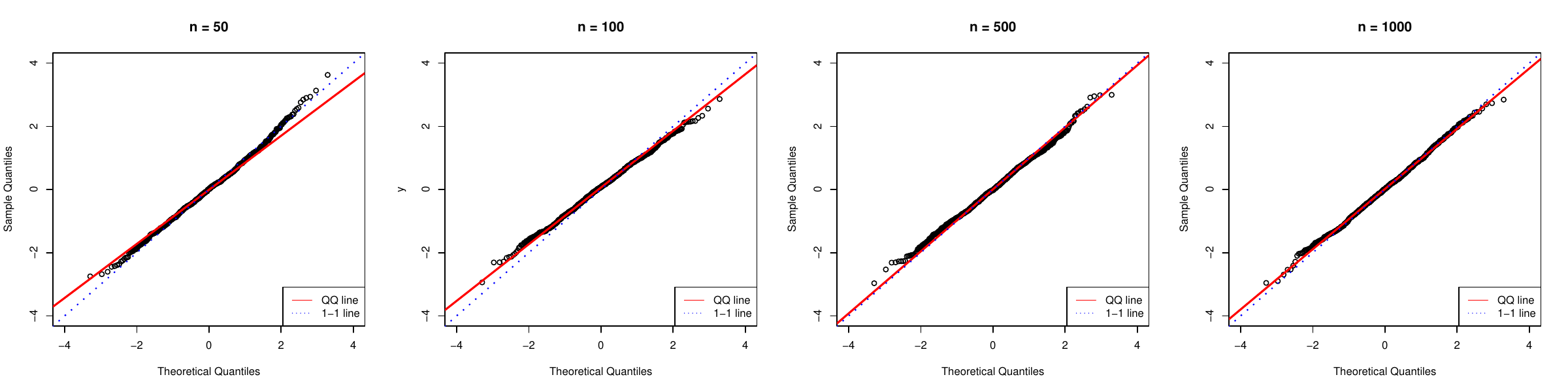}
    \caption{Under the settings of Figure \ref{fig:pos_den_est}, we complement Figure \ref{fig:qq-plots} by showing the QQ-plots for the posterior samples of $\mu$ centered by its asymptotic mean $\mu_0$ and standardized by $(nI_{22})^{-1/2}$ (top panels), and centered by the MLE $\hat{\mu}_n$ and standardized by $(n\hat{I}_{22})^{-1/2}$ (bottom panels). Here, $I_{22}^{-1/2}$ and $\hat{I}_{22}^{-1/2}$ denote the square roots of the second diagonal elements of $I^{-1}(\btheta_0)$ and $\hat{I}^{-1}(\hat{\btheta}_n)$ respectively.}
    \label{fig:qq-plots-mu}
\end{figure}

\begin{figure}
    \centering
    \includegraphics[width=\linewidth]{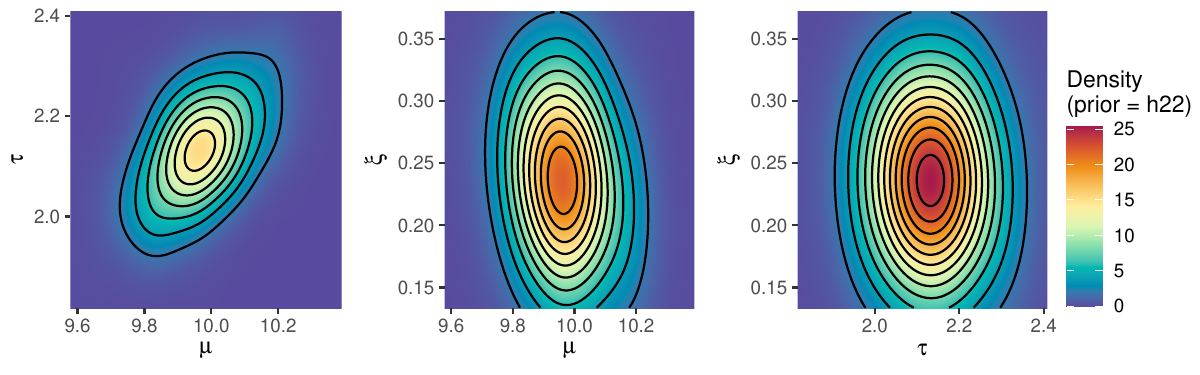}
    \includegraphics[width=\linewidth]{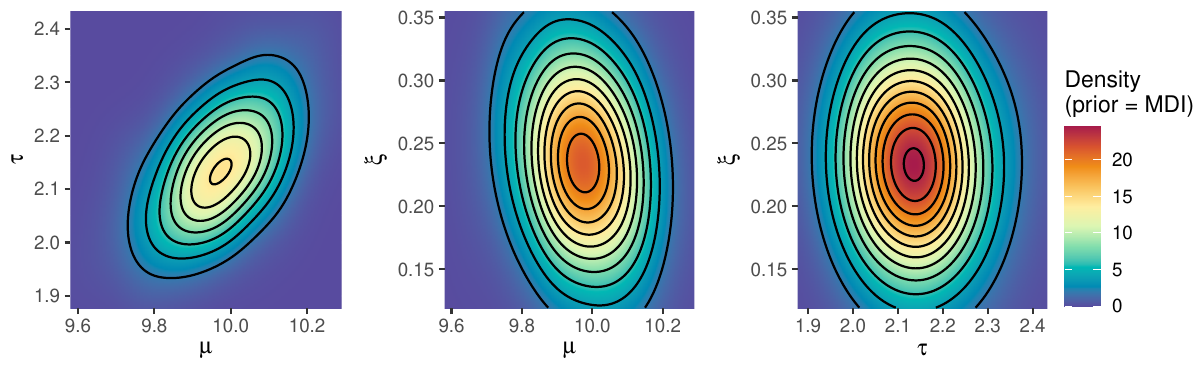}
    \includegraphics[width=\linewidth]{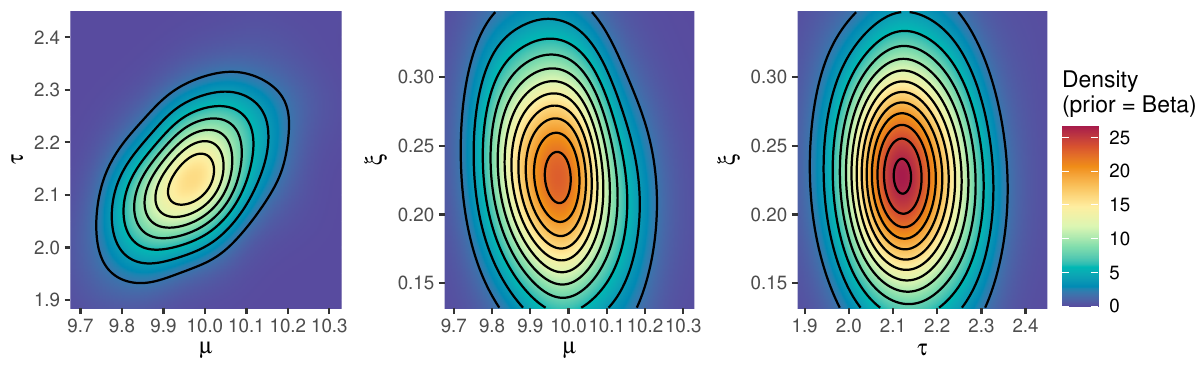}
    \caption{Complementary to Figure \ref{fig:pos_den_est}, 
    we show posterior density estimates based on i.i.d. samples from the GEV distribution with $(\tau_0,\mu_0,\xi_0) = (2, 10, 0.2)$ and sample size $n=500$, under four different priors:   from top to bottom, $\pi(\btheta) \propto \tau^{-1}h^{1/2}_{22}(\xi)$, $\pi(\btheta) \propto \tau^{-1}e^{-\gamma(1+\xi)-1}$ (MDI), and $\pi(\btheta) \propto \tau^{-1}p_{\text{Beta}(9,5)}(\xi+0.5)$ (Beta), in which $p_{\text{Beta}(9,5)}$ is the density of a Beta$(9,5)$ distribution. See  Proposition 2 in \citet{zhang2024reference} for the exact expression of $h_{22}(\xi)$.}
    \label{fig:pos_den_est_prior}
\end{figure}

\begin{figure}
    \centering
    \includegraphics[width=\linewidth]{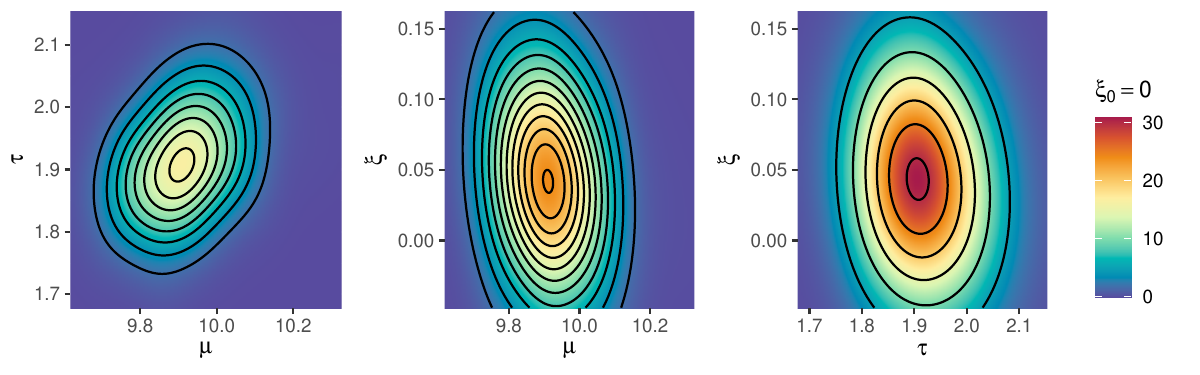}
    \includegraphics[width=\linewidth]{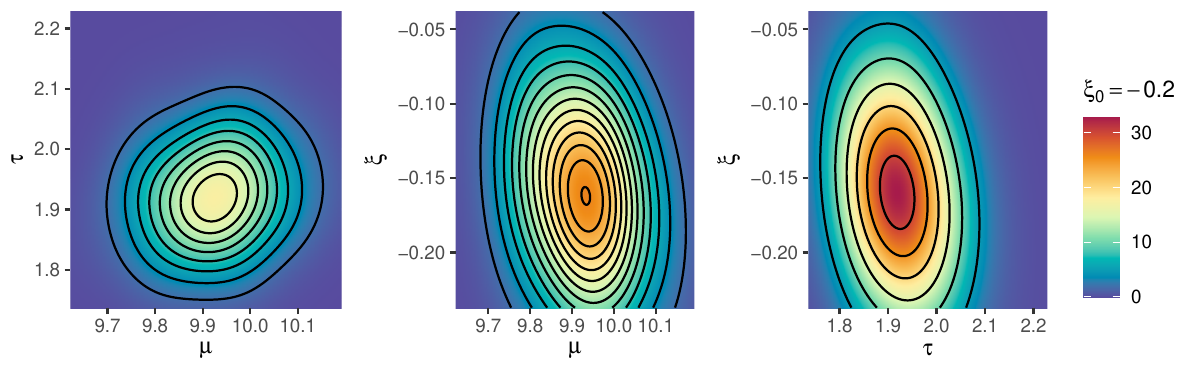}
    \caption{Similar to Figure \ref{fig:pos_den_est}, we show posterior density estimates based on i.i.d. samples from the GEV distribution with $(\tau_0,\mu_0,\xi_0) = (2, 10, 0)$ (top panels) and $(\tau_0,\mu_0,\xi_0) = (2, 10, -0.2)$ (bottom panels). Here, we only show results from $n=500$ and using prior $\pi(\btheta) \propto \tau^{-1}h_{11}(\xi)$.}
\label{fig:zero_neg}
\end{figure}

\begin{figure}
    \centering
    \includegraphics[width=0.825\linewidth]{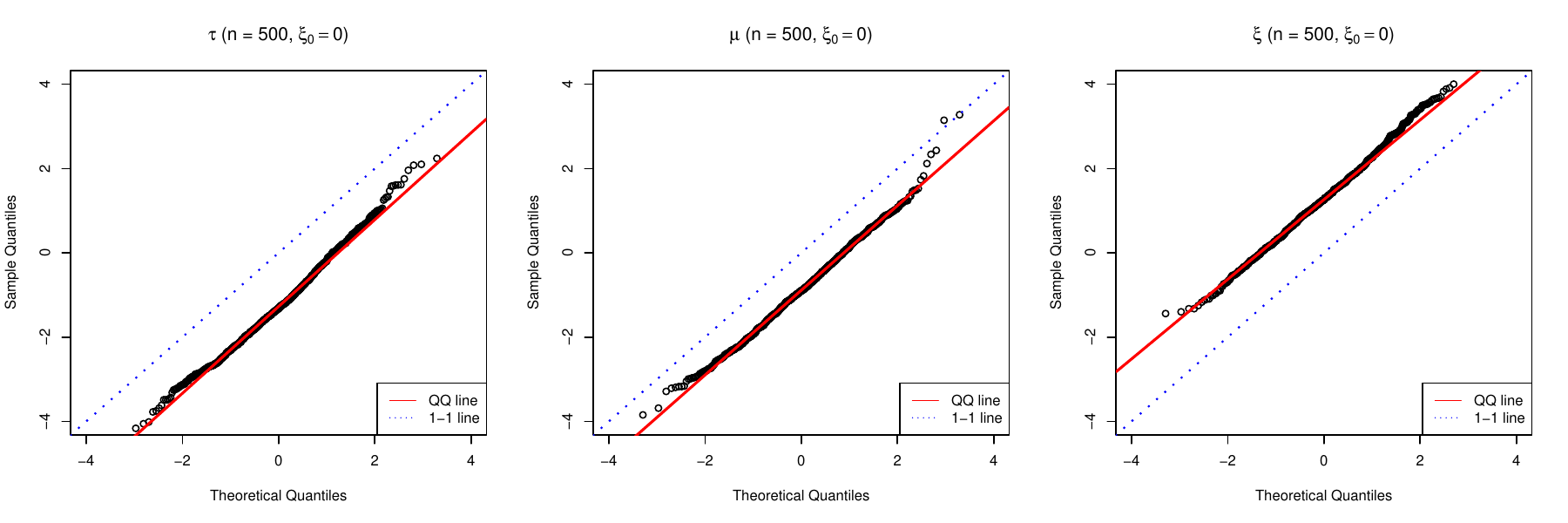}
    \includegraphics[width=0.825\linewidth]{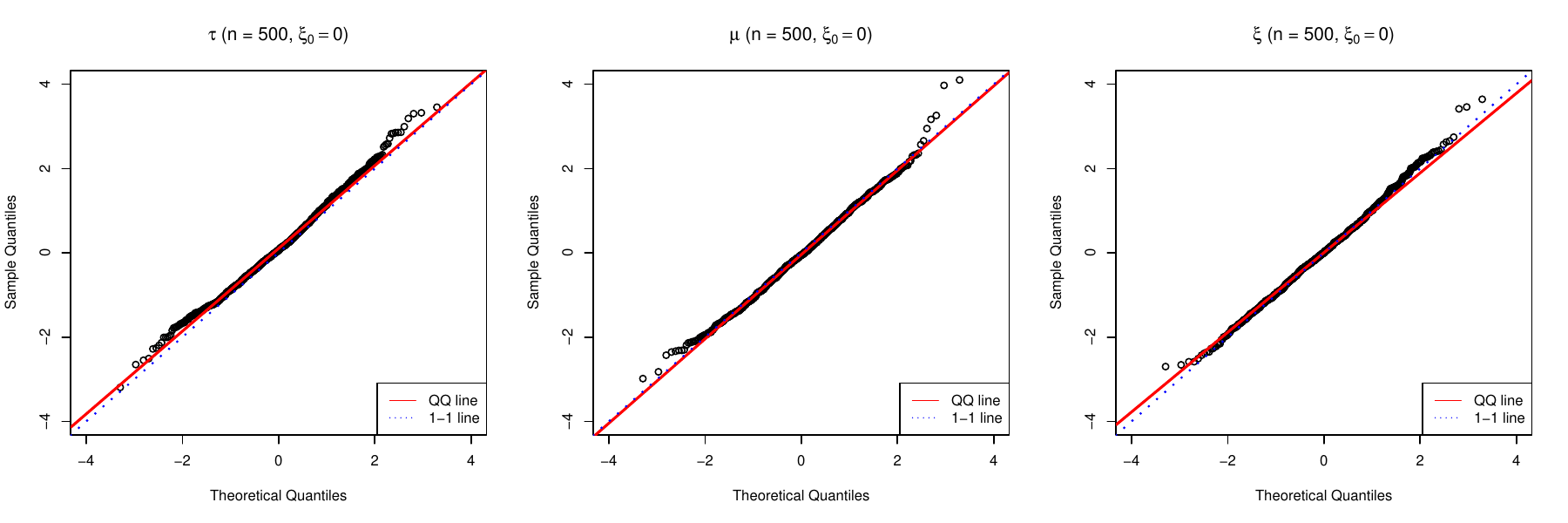}
    \caption{Similar to Figures \ref{fig:qq-plots} -- \ref{fig:qq-plots-mu}, we show   the QQ-plots for the posterior samples of all three parameters (from the 500 i.i.d. GEV samples with $(\tau_0,\mu_0,\xi_0) = (2, 10, 0)$) centered by their asymptotic means and standard deviations (top panels), and centered by the MLEs and standardized by empirical standard deviations (bottom panels).}
    \label{fig:qq_zero}
\end{figure}

\begin{figure}
    \centering
    \includegraphics[width=0.825\linewidth]{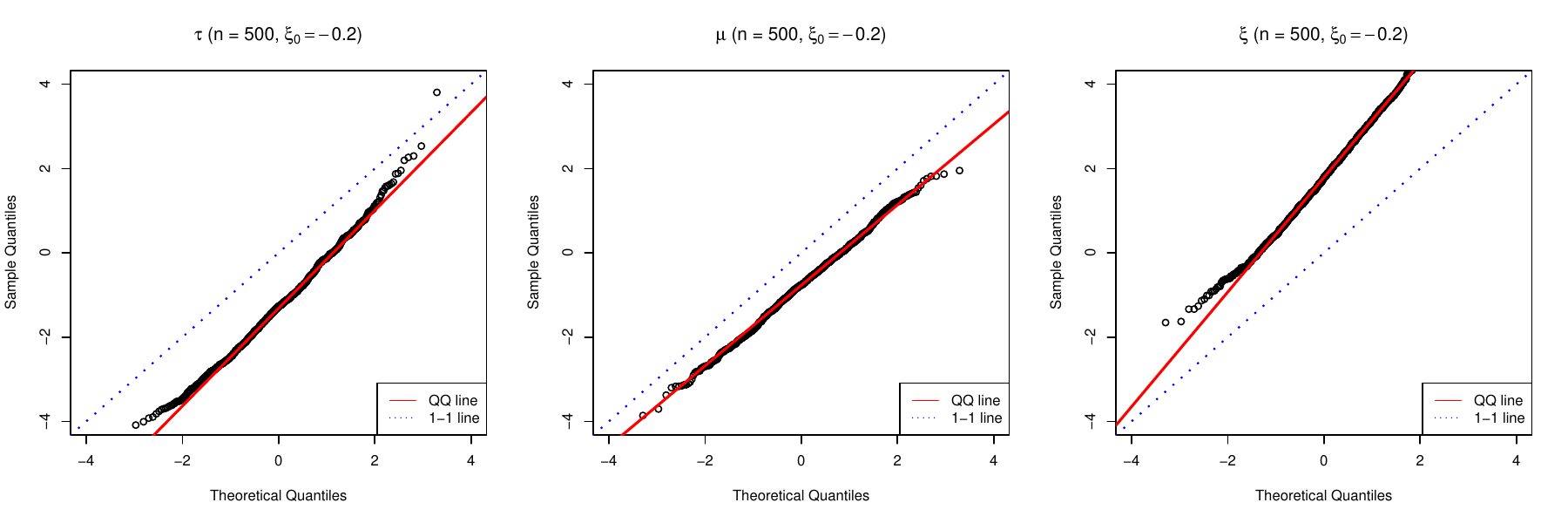}
    \includegraphics[width=0.825\linewidth]{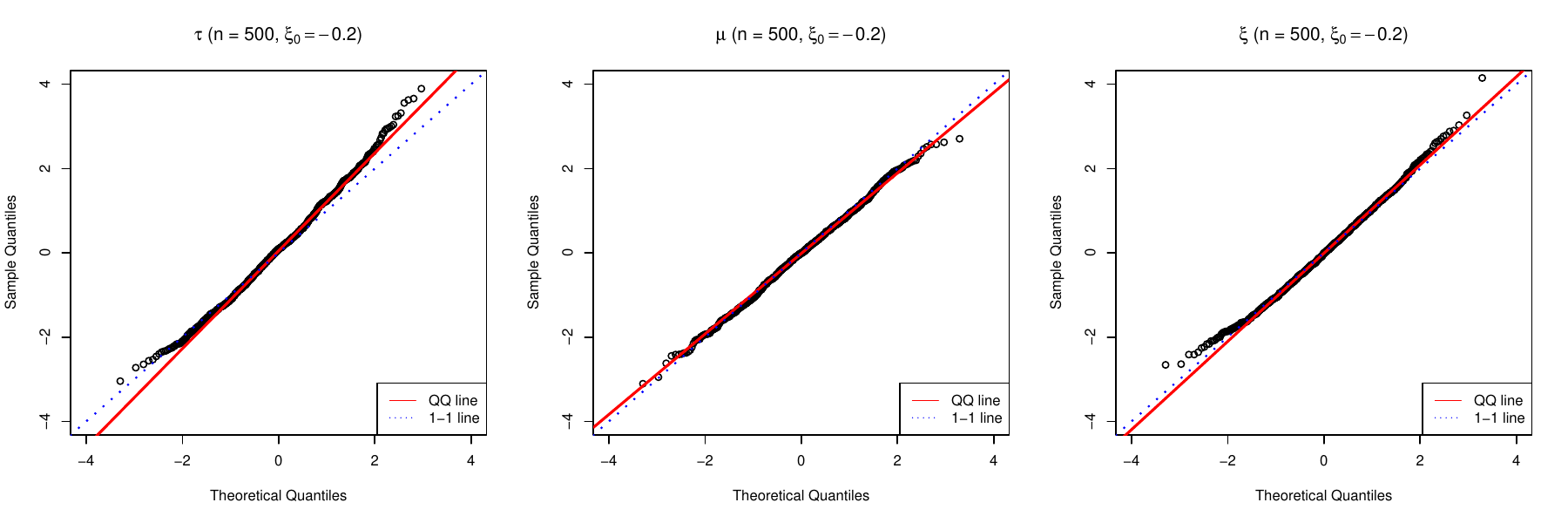}
    \caption{Similar to Figures \ref{fig:qq-plots} -- \ref{fig:qq-plots-mu}, we show   the QQ-plots for the posterior samples of all three parameters (from the 500 i.i.d. GEV samples with $(\tau_0,\mu_0,\xi_0) = (2, 10, -0.2)$) centered by their asymptotic means and standard deviations (top panels), and centered by the MLEs and standardized by empirical standard deviations (bottom panels).}
    \label{fig:qq_neg}
\end{figure}

\newpage

\bibliographystyle{unsrtnat}

\bibliography{main}

\begin{thebibliography}{42}
\providecommand{\natexlab}[1]{#1}
\providecommand{\url}[1]{\texttt{#1}}
\expandafter\ifx\csname urlstyle\endcsname\relax
  \providecommand{\doi}[1]{doi: #1}\else
  \providecommand{\doi}{doi: \begingroup \urlstyle{rm}\Url}\fi

\bibitem[Coles and Tawn(1996)]{coles1996bayesian}
Stuart~G Coles and Jonathan~A Tawn.
\newblock A bayesian analysis of extreme rainfall data.
\newblock \emph{Journal of the Royal Statistical Society: Series C (Applied
  Statistics)}, 45\penalty0 (4):\penalty0 463--478, 1996.

\bibitem[Martins and Stedinger(2000)]{martins2000generalized}
Eduardo~S Martins and Jery~R Stedinger.
\newblock Generalized maximum-likelihood generalized extreme-value quantile
  estimators for hydrologic data.
\newblock \emph{Water Resources Research}, 36\penalty0 (3):\penalty0 737--744,
  2000.

\bibitem[Northrop and Attalides(2016)]{northrop2016posterior}
Paul~J Northrop and Nicolas Attalides.
\newblock Posterior propriety in {B}ayesian extreme value analyses using
  reference priors.
\newblock \emph{Statistica Sinica}, pages 721--743, 2016.
\newblock URL \url{https://doi.org/10.5705/ss.2014.034}.

\bibitem[Fisher and Tippett(1928)]{fisher1928limiting}
Ronald~Aylmer Fisher and Leonard Henry~Caleb Tippett.
\newblock Limiting forms of the frequency distribution of the largest or
  smallest member of a sample.
\newblock In \emph{Mathematical proceedings of the Cambridge philosophical
  society}, volume~24, pages 180--190. Cambridge University Press, 1928.
\newblock URL \url{https://doi.org/10.1017/s0305004100015681}.

\bibitem[Fisher(1922)]{fisher1922mathematical}
Ronald~A Fisher.
\newblock On the mathematical foundations of theoretical statistics.
\newblock \emph{Philosophical Transactions of the Royal Society of London.
  Series A, Containing Papers of a Mathematical or Physical Character},
  222\penalty0 (594-604):\penalty0 309--368, 1922.
\newblock URL \url{https://doi.org/10.1007/978-1-4612-0919-5_2}.

\bibitem[Cram{\'e}r(1946)]{cramer1946mathematical}
Harald Cram{\'e}r.
\newblock Mathematical methods of statistics.
\newblock \emph{Mathematical methods of statistics.}, 1946.

\bibitem[Wald(1949)]{wald1949note}
Abraham Wald.
\newblock Note on the consistency of the maximum likelihood estimate.
\newblock \emph{The Annals of Mathematical Statistics}, 20\penalty0
  (4):\penalty0 595--601, 1949.
\newblock URL \url{https://doi.org/10.1214/aoms/1177729952}.

\bibitem[Bernstein(1927)]{bernstein1927theory}
S~Bernstein.
\newblock Theory of probability, 1927.

\bibitem[Von~Mises(1931)]{von1931wahrscheinlichkeit}
Richard Von~Mises.
\newblock \emph{Wahrscheinlichkeitsrechnung}.
\newblock Springer-Verlag, 1931.

\bibitem[Dombry(2015)]{dombry2015existence}
Cl{\'e}ment Dombry.
\newblock Existence and consistency of the maximum likelihood estimators for
  the extreme value index within the block maxima framework.
\newblock \emph{Bernoulli}, 21\penalty0 (1):\penalty0 420--436, 2015.
\newblock URL \url{https://doi.org/10.3150/13-bej573}.

\bibitem[B{\"u}cher and Segers(2018)]{bucher2018maximum}
Axel B{\"u}cher and Johan Segers.
\newblock Maximum likelihood estimation for the {F}r{\'e}chet distribution
  based on block maxima extracted from a time series.
\newblock \emph{Bernoulli}, 24\penalty0 (2):\penalty0 1427--1462, 2018.
\newblock URL \url{https://doi.org/10.3150/16-bej903}.

\bibitem[Dombry and Ferreira(2019)]{dombry2019maximum}
Cl{\'e}ment Dombry and Ana Ferreira.
\newblock Maximum likelihood estimators based on the block maxima method.
\newblock \emph{Bernoulli}, 25\penalty0 (3):\penalty0 1690--1723, 2019.
\newblock URL \url{https://doi.org/10.3150/18-bej1032}.

\bibitem[B\"{u}cher and Segers(2017)]{bucher2017maximum}
Axel B\"{u}cher and Johan Segers.
\newblock On the maximum likelihood estimator for the generalized extreme-value
  distribution.
\newblock \emph{Extremes}, 20\penalty0 (4):\penalty0 839--872, 2017.
\newblock ISSN 1386-1999.
\newblock URL \url{https://doi.org/10.1007/s10687-017-0292-6}.

\bibitem[Zhang and Shaby(2021)]{zhang2020uniqueness}
Likun Zhang and Benjamin~A Shaby.
\newblock {Uniqueness and global optimality of the maximum likelihood estimator
  for the generalized extreme value distribution}.
\newblock \emph{Biometrika}, 08 2021.
\newblock ISSN 0006-3444.
\newblock URL \url{https://doi.org/10.1093/biomet/asab043}.

\bibitem[Mises(1936)]{mises1936distribution}
R~von Mises.
\newblock La distribution de la plus grande de n valeurs.
\newblock \emph{Rev. Math. Union Interbalcanique}, 1:\penalty0 141--160, 1936.

\bibitem[Jenkinson(1955)]{jenkinson1955frequency}
Arthur~F Jenkinson.
\newblock The frequency distribution of the annual maximum (or minimum) values
  of meteorological elements.
\newblock \emph{Quarterly Journal of the Royal Meteorological Society},
  81\penalty0 (348):\penalty0 158--171, 1955.
\newblock URL \url{https://doi.org/10.1002/qj.49708134804}.

\bibitem[Le~Cam(1958)]{lecam1958proprietes}
Lucien Le~Cam.
\newblock Les propri{\'e}t{\'e}s asymptotiques des solutions de {B}ayes.
\newblock \emph{Publ. Inst. Statist. Univ. Paris}, 7:\penalty0 17--35, 1958.

\bibitem[Freedman(1963)]{freedman1963asymptotic}
David~A Freedman.
\newblock On the asymptotic behavior of {B}ayes' estimates in the discrete
  case.
\newblock \emph{The Annals of Mathematical Statistics}, pages 1386--1403, 1963.
\newblock URL \url{https://doi.org/10.1214/aoms/1177700155}.

\bibitem[Walker(1969)]{walker1969asymptotic}
Andrew~M Walker.
\newblock On the asymptotic behaviour of posterior distributions.
\newblock \emph{Journal of the Royal Statistical Society: Series B
  (Methodological)}, 31\penalty0 (1):\penalty0 80--88, 1969.
\newblock URL \url{https://doi.org/10.1111/j.2517-6161.1969.tb00767.x}.

\bibitem[Bickel and Yahav(1969)]{bickel1969some}
Peter~J Bickel and Joseph~A Yahav.
\newblock Some contributions to the asymptotic theory of {B}ayes solutions.
\newblock \emph{Zeitschrift f{\"u}r Wahrscheinlichkeitstheorie und verwandte
  Gebiete}, 11\penalty0 (4):\penalty0 257--276, 1969.
\newblock URL \url{https://doi.org/10.1007/bf00531650}.

\bibitem[Chao(1970)]{chao1970asymptotic}
MT~Chao.
\newblock The asymptotic behavior of {B}ayes' estimators.
\newblock \emph{The Annals of Mathematical Statistics}, 41\penalty0
  (2):\penalty0 601--608, 1970.
\newblock URL \url{https://doi.org/10.1214/aoms/1177697100}.

\bibitem[{v}an~der Vaart(2000)]{van2000asymptotic}
Aad~W {v}an~der Vaart.
\newblock \emph{Asymptotic statistics}, volume~3.
\newblock Cambridge university press, 2000.

\bibitem[Zellner(1971)]{zellner1971introduction}
Arnold Zellner.
\newblock \emph{An introduction to {B}ayesian inference in econometrics}.
\newblock John Wiley \& Sons, Inc., New York-London-Sydney, 1971.
\newblock Wiley Series in Probability and Mathematical Statistics.

\bibitem[Jeffreys(1961)]{jeffreys1961theory}
Harold Jeffreys.
\newblock \emph{Theory of probability}.
\newblock Third edition. Clarendon Press, Oxford, 1961.

\bibitem[Zhang and Shaby(2024)]{zhang2024reference}
Likun Zhang and Benjamin~A Shaby.
\newblock Reference priors for the generalized extreme value distribution.
\newblock \emph{Statistica Sinica}, 2024.
\newblock URL \url{https://doi.org/10.5705/ss.202021.0258}.

\bibitem[Berger et~al.(2009)Berger, Bernardo, and Sun]{berger2009formal}
James~O Berger, Jos{\'e}~M Bernardo, and Dongchu Sun.
\newblock The formal definition of reference priors.
\newblock \emph{The Annals of Statistics}, 37\penalty0 (2):\penalty0 905--938,
  2009.

\bibitem[Dawid(1970)]{dawid1970limiting}
AP~Dawid.
\newblock On the limiting normality of posterior distributions.
\newblock In \emph{Mathematical proceedings of the Cambridge philosophical
  society}, volume~67, pages 625--633. Cambridge University Press, 1970.
\newblock URL \url{https://doi.org/10.1017/s0305004100045953}.

\bibitem[Heyde and Johnstone(1979)]{heyde1979asymptotic}
CC~Heyde and IM~Johnstone.
\newblock On asymptotic posterior normality for stochastic processes.
\newblock \emph{Journal of the Royal Statistical Society: Series B
  (Methodological)}, 41\penalty0 (2):\penalty0 184--189, 1979.
\newblock URL \url{https://doi.org/10.1007/978-1-4419-5823-5_45}.

\bibitem[Chen(1985)]{chen1985asymptotic}
Chan-Fu Chen.
\newblock On asymptotic normality of limiting density functions with {B}ayesian
  implications.
\newblock \emph{Journal of the Royal Statistical Society: Series B
  (Methodological)}, 47\penalty0 (3):\penalty0 540--546, 1985.
\newblock URL \url{https://doi.org/10.1111/j.2517-6161.1985.tb01384.x}.

\bibitem[Ghosal et~al.(1995)Ghosal, Ghosh, Samanta,
  et~al.]{ghosal1995convergence}
Subhashis Ghosal, Jayanta~K Ghosh, Tapas Samanta, et~al.
\newblock On convergence of posterior distributions.
\newblock \emph{The Annals of Statistics}, 23\penalty0 (6):\penalty0
  2145--2152, 1995.
\newblock URL \url{https://doi.org/10.1214/aos/1034713651}.

\bibitem[Diaconis and Freedman(1986)]{diaconis1986consistency}
Persi Diaconis and David Freedman.
\newblock On the consistency of {B}ayes estimates.
\newblock \emph{The Annals of Statistics}, pages 1--26, 1986.
\newblock URL \url{https://doi.org/10.1214/aos/1176349830}.

\bibitem[Olver et~al.(2010)Olver, Lozier, Boisvert, and Clark]{olver2010nist}
Frank~WJ Olver, Daniel~W Lozier, Ronald~F Boisvert, and Charles~W Clark.
\newblock \emph{NIST handbook of mathematical functions hardback and CD-ROM}.
\newblock Cambridge university press, 2010.

\bibitem[Shaby and Wells(2010)]{shaby-2010a}
Benjamin Shaby and Martin Wells.
\newblock Exploring an adaptive {M}etropolis algorithm.
\newblock Technical Report 1011-14, Duke University Department of Stastical
  Science, 2010.

\bibitem[Rizzo et~al.(2018)Rizzo, Sz{\'e}kely, et~al.]{rizzo2018statistics}
ML~Rizzo, G~Sz{\'e}kely, et~al.
\newblock E-statistics: Multivariate inference via the energy of data.
\newblock \emph{R package “energy,” version 1.7}, 5, 2018.

\bibitem[Stein(1956)]{stein1956inadmissibility}
Charles Stein.
\newblock Inadmissibility of the usual estimator for the mean of a multivariate
  normal distribution.
\newblock Technical report, Stanford University Stanford United States, 1956.
\newblock URL \url{https://doi.org/10.1525/9780520313880-018}.

\bibitem[Dawid et~al.(1973)Dawid, Stone, and Zidek]{dawid1973marginalization}
A~Philip Dawid, Mervyn Stone, and James~V Zidek.
\newblock Marginalization paradoxes in {B}ayesian and structural inference.
\newblock \emph{Journal of the Royal Statistical Society: Series B
  (Methodological)}, 35\penalty0 (2):\penalty0 189--213, 1973.
\newblock URL \url{https://doi.org/10.1111/j.2517-6161.1973.tb00952.x}.

\bibitem[Resnick(2008)]{resnick2008extreme}
Sidney~I Resnick.
\newblock \emph{Extreme values, regular variation, and point processes},
  volume~4.
\newblock Springer Science \& Business Media, 2008.

\bibitem[Carlson(1977)]{carlson1977special}
Bille~Chandler Carlson.
\newblock \emph{Special functions of applied mathematics}.
\newblock Academic Press, 1977.

\bibitem[Carlson(1966)]{carlson1966some}
Bille~C Carlson.
\newblock Some inequalities for hypergeometric functions.
\newblock \emph{Proceedings of the American Mathematical Society}, 17\penalty0
  (1):\penalty0 32--39, 1966.
\newblock URL \url{https://doi.org/10.1090/s0002-9939-1966-0188497-6}.

\bibitem[Temme(1996)]{temme1996uniform}
NM~Temme.
\newblock Uniform asymptotics for the incomplete gamma functions starting from
  negative values of the parameters.
\newblock \emph{Methods and Applications of Analysis}, 3\penalty0 (3):\penalty0
  335--344, 1996.
\newblock URL \url{https://doi.org/10.4310/maa.1996.v3.n3.a3}.

\bibitem[Chung(1949)]{chung1949estimate}
Kai-Lai Chung.
\newblock An estimate concerning the kolmogroff limit distribution.
\newblock \emph{Transactions of the American Mathematical Society}, 67\penalty0
  (1):\penalty0 36--50, 1949.
\newblock URL \url{https://doi.org/10.2307/1990415}.

\bibitem[Mitrinovi{\'c}(1964)]{dragoslav1964elementary}
D.S. Mitrinovi{\'c}.
\newblock Elementary inequalities.
\newblock \emph{P. Noordhoff Ltd, Groningen}, 1964.
\newblock URL \url{https://doi.org/10.1017/s0008439500027995}.

\end{thebibliography}

\end{document}